\documentclass{amsart}
\usepackage{amsmath,amssymb,amscd,latexsym}
\usepackage{enumerate,verbatim,calc}
\usepackage{rotate}
\usepackage[all]{xy}
\SelectTips{eu}{}
\xyoption{curve}

\usepackage[colorlinks]{hyperref}

\newcommand{\refl}[1][C]{{\mathsf R}_{#1}}

\newcommand{\ts}{\textstyle}

\newcommand{\sss}{\scriptscriptstyle}

\newcommand{\ges}{{\sss\geqslant}}

\newcommand{\ann}{\operatorname{Ann}}

\newcommand{\bwedge}{{\textstyle\bigwedge}}

\newcommand{\pic}[1]{\operatorname{Pic}(#1)}
\newcommand{\dpic}[1]{\operatorname{DPic}(#1)}

\newcommand{\ssd}{{\mathsf d}}

\newcommand{\sh}{{\mathsf h}}

\newcommand{\sA}{{\mathsf A}}

\newcommand{\fm}{{\mathfrak m}}
\newcommand{\fn}{{\mathfrak n}}
\newcommand{\fp}{{\mathfrak p}}
\newcommand{\fq}{{\mathfrak q}}

\newcommand{\amp}{\operatorname{amp}}

\newcommand{\col}{\colon}
\newcommand{\dd}{\partial}
\newcommand{\depth}{\operatorname{depth}}

\newcommand{\ED}[3]{{}^{#1}{\operatorname{d}}^{#2,#3}}

\newcommand{\EH}[3]{{}^{#1}\!\operatorname{E}_{#2,#3}}

\newcommand{\gdim}{\operatorname{G-dim}}

\newcommand{\fd}{\operatorname{fd}}
\newcommand{\rfd}[2]{\operatorname{Rfd}_{#1}{#2}}

\newcommand{\hh}[1]{\operatorname{H}(#1)}

\newcommand{\HH}[2]{\operatorname{H}_{#1}(#2)}

\newcommand{\hra}{\hookrightarrow}
\newcommand{\id}{\operatorname{id}}
\newcommand{\image}{\operatorname{Im}}

\newcommand{\lra}{\longrightarrow}
\newcommand{\Max}{\operatorname{Max}}

\newcommand{\ord}{\operatorname{ord}}

\newcommand{\pd}{\operatorname{pd}}

\newcommand{\rank}{\operatorname{rank}}

\newcommand{\Shift}{\mathsf{\Sigma}}
\newcommand{\Spec}{\operatorname{Spec}}

\newcommand{\Supp}{\operatorname{Supp}}

\newcommand{\Ext}[4]{\operatorname{Ext}^{#1}_{#2}(#3,#4){}}
\newcommand{\Hom}[3]{\operatorname{Hom}_{#1}(#2,#3)}
\newcommand{\Rhom}[3]{\operatorname{\mathsf{R}Hom}_{#1}(#2,#3)}
\newcommand{\dtensor}[1]{\otimes^{\mathsf{L}}_{#1}}
\newcommand{\Tor}[4]{\operatorname{Tor}_{#1}^{#2}(#3,#4){}}
\newcommand{\dcat}[1][S]{{\mathsf D}(#1)}
\newcommand{\dcatB}[1]{{\mathsf{D}_{\mathsf b}}(#1)}
\newcommand{\dcatb}[1]{{\mathsf{D}^{\mathsf f}_{\mathsf b}}(#1)}

\newcommand{\dcatN}[1]{{\mathsf{D}_{\!{\scriptscriptstyle\mathsf -}}} 
(#1)}
\newcommand{\dcatP}[1]{{\mathsf{D}_{\!{\scriptscriptstyle\mathsf +}}} 
(#1)}
\newcommand{\dcatn}[1]{{\mathsf{D}^{\mathsf f}_{\!{\scriptscriptstyle 
\mathsf -}}}(#1)}
\newcommand{\dcatp}[1]{{\mathsf{D}^{\mathsf f}_{\!{\scriptscriptstyle 
\mathsf +}}}(#1)}

\newcommand{\dcatfg}[1]{{\mathsf{D}^{\mathsf f}}(#1)}

\newcommand{\tra}{\twoheadrightarrow}

\newcommand{\vf}{{\varphi}}

\newcommand{\wt}{\widetilde}

\newcommand{\xra}{\xrightarrow}

\newcommand{\ZZ}{\operatorname{Z}}

\newcommand{\BN}{{\mathbb N}}

\newcommand{\BZ}{{\mathbb Z}}

\theoremstyle{plain}

\newtheorem{theorem}{Theorem}[section]
\newtheorem{proposition}[theorem]{Proposition}
\newtheorem{lemma}[theorem]{Lemma}
\newtheorem{corollary}[theorem]{Corollary}

\newtheorem{itheorem}{Theorem}
\newtheorem{subcorollary}{Corollary}[theorem]
\newtheorem{sublemma}[subcorollary]{Lemma}

\newtheorem{ssubtheorem}{Theorem}[subsection]
\newtheorem{ssubproposition}[ssubtheorem]{Proposition}

\newtheorem{ssubcorollary}[ssubtheorem]{Corollary}

\theoremstyle{definition}

\newtheorem{example}[theorem]{Example}

\newtheorem{ssubchunk}[ssubtheorem]{}

\newtheorem{chunk}[theorem]{}
\newtheorem{subchunk}[subcorollary]{}

\theoremstyle{remark}

\newtheorem{remark}[theorem]{Remark}

\newtheorem{notes}[theorem]{Notes}

\newtheorem{ssubnotes}[ssubtheorem]{Notes}

\newtheorem{ssubquestion}[ssubtheorem]{Question}

\numberwithin{equation}{theorem}

\newcommand\iso{{\mkern8mu\longrightarrow \mkern-25.5mu{}^\sim \mkern17mu}}
\newcommand\tiso{{\mkern6mu\rightarrow\mkern-21mu{\vphantom{h}}^\simeq \mkern9mu}}

\begin{document}

\title[Reflexivity and rigidity. I]{Reflexivity and rigidity for complexes\\ 
I. Commutative rings}

\author[L.\,L.\,Avramov]{Luchezar L.~Avramov} 
\address{Department of Mathematics, University of Nebraska, Lincoln, NE 68588, U.S.A.}
\email {avramov@math.unl.edu}

\author[S.\,B.\,Iyengar]{Srikanth B.~Iyengar} 
\address{Department of Mathematics,
University of Nebraska, Lincoln, NE 68588, U.S.A.}
\email{iyengar@math.unl.edu}

\author[J.\,Lipman]{Joseph Lipman} 
\address{Department of Mathematics, Purdue University, W. Lafayette, IN 47907, U.S.A.}
\email{jlipman@purdue.edu}

\thanks{Research partly supported by NSF grant DMS 0803082 (LLA),
NSF grant DMS 0602498 (SBI), and NSA grant H98230-06-1-0010 (JL)}



\dedicatory{To our friend and colleague, Hans-Bj{\o}rn Foxby.}


 \begin{abstract}
A notion of rigidity with respect to an arbitrary semidualizing complex $C$ over 
a commutative noetherian ring  $R$ is introduced and studied.  One of the main 
results characterizes $C$-rigid complexes. Specialized to the case 
when $C$ is the relative dualizing complex of a homomorphism of rings of finite 
Gorenstein dimension, it leads to broad generalizations of theorems of Yekutieli 
and Zhang concerning rigid dualizing complexes, in the sense of Van den Bergh. Along the way, 
new results about derived reflexivity with respect to $C$ are established. 
Noteworthy is the statement that derived $C$-reflexivity is a local 
property; it implies that a finite $R$-module $M$ has finite $G$-dimension over 
$R$ if $M_{\fm}$ has finite $G$-dimension over $R_{\fm}$ for each maximal 
ideal $\fm$ of $R$.
 \end{abstract}

\maketitle

\setcounter{tocdepth}{2}
\tableofcontents

\section*{Introduction} 

Rigidification means, roughly, endowing a type of object with 
extra structure so as to eliminate nonidentity automorphisms.  
For example, a rigidification for dualizing sheaves on varieties over 
perfect fields plays an important role in \cite{Lp1.5}.  We will be concerned 
with rigidifying  complexes arising from Grothendieck duality theory, both 
in commutative algebra and in algebraic geometry.  This  paper is 
devoted to the algebraic situation; the geometric counterpart is treated 
in \cite{AIL}.

Let $R$ be a noetherian ring and $\dcat[R]$ its derived category. We write 
$\dcatb R$ for the full subcategory of \emph{homologically finite complexes,} 
that is to say, complexes $M$ for which the $R$-module $\hh M$ is finitely 
generated. Given complexes $M$ and $C$ in~$\dcatb R$ one says that $M$
is \emph{derived $C$-reflexive} if the canonical map 
 \[
\delta^C_M\col M\lra\Rhom R{\Rhom RMC}C
 \]
is an isomorphism and $\Rhom RMC$ is homologically finite.  When the 
ring $R$ has finite Krull dimension, the complex $C$~is said to be 
\emph{dualizing for} $R$ if $\delta^C_M$ is an isomorphism for \emph{all} 
homologically finite complexes $M$. In \cite[p.\,258, 2.1]{H} it is proved  
that when $C$ is isomorphic to some bounded complex of injective modules, $C$
is dualizing if and only if it is \emph{semidualizing,} meaning that the canonical map 
 \[
\chi^C\col R\lra\Rhom RCC
 \]
is an isomorphism.

Even when $\Spec R$ is connected, dualizing complexes for $R$ differ by shifts 
and the action of the Picard group of the ring \cite[p.\,266, 3.1]{H}. Such a lack of 
uniqueness has been a source of difficulties.  Building on work of Van den Bergh 
\cite{VdB} and extensively using differential graded algebras, in \cite{YZ1,YZ2} 
Yekutieli and Zhang have developed for algebras of finite type over a regular ring $K$
of finite Krull dimension a theory of \emph{rigid relative to $K$} dualizing complexes.  
The additional structure that they carry makes them unique up to unique rigid 
isomorphism.

Our approach to rigidity applies to any noetherian ring $R$ and takes place 
entirely within its derived category:  We say that $M$ is \emph{$C$-rigid} 
if there is an isomorphism 
 \[
\mu\col M\xra{\,\simeq\,}\Rhom R{\Rhom RMC}M\,,
 \]
called a $C$-rigidifying isomorphism for $M$.  In the context described in
the preceding paragraph we prove, using the main result of \cite{AILN},
that rigidity in the sense of Van den Bergh, Yekutieli, and Zhang coincides
with $C$-rigidity for a specific complex $C$.

The precise significance of $C$-rigidity is explained by the following result.  
It is abstracted from Theorem \ref{thm:uniquerigidity}, which requires no 
connectedness hypothesis.

   \begin{itheorem}
     \label{ithm:rigid}
If $C$ is a semidualizing complex, then $\Rhom R{\chi^C}C^{-1}$ is a 
$C$-rigidifying isomorphism.

When $\Spec R$ is connected and $M$ is non-zero and $C$-rigid, 
with $C$-rigidifying isomorphism $\mu$, there exists a unique isomorphism 
$\alpha\col C\iso M$ making the following diagram commute:
 \[
\xymatrixrowsep{2.5pc}
\xymatrixcolsep{6.7pc}
\xymatrix{ 
C\ar@{->}[r]^-{\Rhom R{\chi^C}C^{-1}}
       \ar@{->}[d]_{\alpha}
  &\, \Rhom R{\Rhom R{C}C}{C} \ar@{->}[d]^{\Rhom R{\Rhom R{\alpha}C}{\>\alpha}}
  \\
M\ar@{->}[r]_-{\mu} 
&\,\Rhom R{\Rhom RMC}M
}
\]
 \end{itheorem}

Semidualizing complexes, identified by Foxby \cite{Fx0} and Golod \cite{Gol} 
in the case of modules, have received considerable attention in \cite{AF:qG} 
and in the work of Christensen, Frankild, Sather-Wagstaff, and Taylor 
\cite{Ch2, FS, FST}.  However, to achieve our goals we need to go further 
back and rethink basic propositions concerning derived reflexivity.  This is 
the content of Sections \ref{Depth} through \ref{Duality}, from where we 
highlight some results.  

   \begin{itheorem}
     \label{ithm:reflexive}
When $C$ is semidualizing, $M$ is derived $C$-reflexive if $(\mkern-1mu$and only if\/$)$ 
there exists some isomorphism $M\simeq\Rhom R{\Rhom RMC}C$ in $\dcat[R],$ 
if $(\mkern-1mu$and only if\/$)$ $M_{\fm}$ is derived $C_{\fm}$-reflexive for each maximal 
ideal $\fm$ of $R$.
 \end{itheorem}

This is part of Theorem \ref{thm:bounded-reflexive}. One reason for its significance 
is that it delivers derived $C$-reflexivity bypassing a delicate step, the verification 
that $\Rhom RMC$ is homologically finite.  Another is that it establishes that derived 
$C$-reflexivity is a local property.  This implies, in particular, that a finite $R$-module 
$M$ has finite \text{G-dimension} (Gorenstein dimension) in the sense of Auslander 
and Bridger \cite{AB} if it has that property at each maximal ideal of $R$; see 
Corollary~\ref{cor:localGdim}.
%


In Theorem \ref{thm:reflsemid} we characterize pairs of mutually reflexive 
complexes:

\begin{itheorem}
\label{ithm:reflsemid}
The complexes $C$ and $M$ are semidualizing and satisfy
$C\simeq L\otimes_RM$ for some invertible graded $R$-module $L$ 
if and only if $M$ is derived $C$-reflexive, $C$~is derived $M$-reflexive, 
and $\hh{M}_\fp\ne0$ holds for every $\fp\in\Spec R$.
 \end{itheorem}

In the last section we apply our results to the \emph{relative 
dualizing complex}~$D^\sigma$ attached to an 
algebra $\sigma\col K\to S$ essentially of finite type over a noetherian ring~$K$; see 
\cite[1.1 and 6.2]{AILN}.  We show that $D^\sigma$ is semidualizing if and 
only if $\sigma$ has finite G-dimension in the sense of \cite{AF:qG}.  
One case when the G-dimension of $\sigma$ is finite is if $S$ has finite 
flat dimension as $K$-module.  In this context, a result of \cite{AILN} implies 
that $D^{\sigma}$-rigidity is equivalent to rigidity relative to $K$, in the sense 
of \cite{YZ2}. We prove:

\begin{itheorem}
\label{ithm:YZ}
If $K$ is Gorenstein, the flat dimension of the $K$-module $S$ is finite, and 
$\dim S$ is finite, then $D^\sigma$ is dualizing for $S$ and is rigid relative to $K$. 

When moreover $\Spec S$ is connected, $D^{\sigma}$ is the unique, up to unique 
rigid isomorphism, non-zero complex in $\dcatb S$ that is rigid relative to $K$.
  \end{itheorem}

This result, which is contained in Theorem~\ref{thm:relativerigidityFD},
applies in particular when $K$ is regular, and is a broad generalization 
of one of the main results in \cite{YZ2}.

Our terminology and notation are mostly in line with literature in commutative 
algebra.  In particular, we put ``homological'' gradings on complexes, so at first 
sight some formulas may look unfamiliar to experts used to cohomological
conventions.  More details may be found in Appendix~\ref{sec:pbseries}, where
we also prove results on Poincar\'e series and Bass series of complexes 
invoked repeatedly in the body of the text.

We are grateful to Lars Winther Christensen, Amnon Neeman, and Sean Sather-Wagstaff for their comments and suggestions on earlier versions of this article.

\medskip

\centerline{***}

Several objects studied in this paper were introduced by Hans-Bj{\o}rn 
Foxby, and various techniques used below were initially developed by him.  We 
have learned a lot about the subject from his articles, his lectures, and through 
collaborations with him.  This work is dedicated to him in appreciation and friendship.

\section{Depth}
\label{Depth}
Throughout the paper \emph{$R$ denotes a commutative noetherian ring}. An $R$-module is said to be `finite' if it can be generated, as an $R$-module, by finitely many elements.

The \emph{depth} of a complex $M$ over a local ring $R$ with residue field 
$k$ is the number
  \[
\depth_{R}M=\inf\{n\in\ZZ\mid \Ext nRkM\ne 0\}\,.
  \]

We focus on a global invariant that appears in work of Chouinard and Foxby:
  \begin{equation}
   \label{eq:rfd}
\rfd RM = \sup\{\depth R_{\fp} - \depth_{R_{\fp}}M_{\fp}\mid \fp\in  \Spec R\}\,.
  \end{equation}
See \ref{notes:rfd} for a different description of this number. Our goal is to prove:

\begin{theorem}
\label{thm:rfd}
Every complex $M$ in $\dcatb R$ satisfies $\rfd RM<\infty$.
\end{theorem}

The desired inequality is obvious for rings of finite Krull dimension. To
handle the general case, we adapt the proof of a result of Gabber,
see Proposition~\ref{prop:gabber}.

A couple of simple facts are needed to keep the argument going:

\begin{chunk}
 \label{rfd:sequence}
If $0\to L\to M\to N\to0$ is an exact sequence of complexes 
then one has
  \[
\rfd RM\le\max\{\rfd RL,\rfd RN\}\,.
  \]

Indeed, for every $\fp\in\Spec R$ and each $n\in\BZ$ one has an  
induced exact sequence
  \[
\Ext n{R_\fp}{R_\fp/\fp R_\fp}{L_\fp}\to \Ext n{R_\fp}{R_\fp/\fp R_ 
\fp}{M_\fp}
\to \Ext n{R_\fp}{R_\fp/\fp R_\fp}{N_\fp}
  \]
that yields $\depth_{R_\fp}\!{M_\fp}\ge\min\{\depth_{R_\fp}\!{L_\fp}, 
\depth_{R_\fp}\!{N_\fp}\}$.
  \end{chunk}

The statement below is an Auslander-Buchsbaum Equality for complexes:

\begin{chunk}
 \label{abe}
Each bounded complex $F$ of finite free modules over a local ring $R$ has
  \[
\depth_{R}F = \depth R - \sup\hh{k\otimes_{R}F}\,,
  \]
see \cite[3.13]{Fx1}.  This formula is an immediate consequence of  
the isomorphisms
\[
\Rhom RkF\simeq \Rhom RkR\dtensor RF\simeq
\Rhom RkR\dtensor{k}(k\otimes_{R}F)
\]
 in $\dcat [R]$, where the first one holds because $F$ is finite free.
\end{chunk}

\begin{proof}[Proof of Theorem \emph{\ref{thm:rfd}}]
It's enough to prove that $\rfd RM<\infty$ holds for cyclic modules.
Indeed, replacing $M$ with a quasi-isomorphic complex we may assume 
$\amp M=\amp\hh M$.  If one has $\amp M=0$, then $M$ is a shift of a 
finite $R$-module, so an induction on the number of its generators,
using \ref{rfd:sequence}, shows that $\rfd RM$ is finite.  Assume  the 
statement holds for all complexes of a given amplitude.  Since 
$L=\Shift^iM_i$ with $i=\inf M$ is a subcomplex of $M$, and one has 
$\amp(M/L)<\amp M$, using \ref{rfd:sequence} and induction we obtain 
$\rfd RM\le\max\{\rfd RL,\rfd R(M/L)\}<\infty$.

By way of contradiction, assume $\rfd R{(R/J)} = \infty$ holds for
some ideal $J$ of $R$.  Since $R$ is noetherian, we may choose $J$ so
that $\rfd R{(R/I)}$ is finite for each ideal $I$ with $I\supsetneq J$.
The ideal $J$ is prime:  otherwise one would have an exact sequence
  \[
0\to R/J' \to R/J\to R/I\to 0\,,
  \]
where $J'$ is a prime ideal associated to $R/J$ with $J'\supsetneq J$;
this implies $I\supsetneq J$, so in view of \ref{rfd:sequence} the exact
sequence yields $\rfd R{(R/J)}<\infty$, which is absurd.

Set $S=R/J$, fix a finite generating set of $J$, let $g$ denote its
cardinality, and $E$ be the Koszul complex on it. As $S$ is a domain
and $\bigoplus_i\HH iE$ is a finite $S$-module, we may choose $f\in
R\smallsetminus J$ so that each $S_{f}$-module ${\HH iE}{}_{f}$
is free.  Now $(J,f)\supsetneq J$ implies that $j=\rfd R(R/(J,f))$ is finite.
To get the desired contradiction we prove
\[
\depth R_{\fp} - \depth_{R_{\fp}}S_{\fp} \le \max\{j-1,g\}
\quad\text{for each}\quad \fp\in\Spec R\,.
\]

In case $\fp\not\supseteq J$ one has $\depth_{R_{\fp}}S_\fp=\infty$, so the
inequality obviously holds. 

When $\fp\supseteq(J,f)$ the exact sequence
  \[
0\to S\xra{f}S\to R/(J,f)\to 0
  \]
yields $\depth_{R_{\fp}}S_{\fp} = \depth_{R_{\fp}}(R/(J,f))_{\fp}+1$,
and hence one has
\[
\depth R_{\fp} - \depth_{R_{\fp}}S_{\fp} \le j-1\,.
\]

It remains to treat the case $f\notin\fp\supseteq J$.  Set $k=R_\fp/\fp
R_\fp$, $d=\depth_{R_\fp}\!S_\fp$, and $s=\sup\hh{E_\fp}$. In the second
quadrant spectral sequence
   \begin{align*}
\EH 2pq= \Ext {-p}{R_\fp}k{\HH q{E_\fp}}
  & \Longrightarrow \Ext {-p-q}{R_\fp}k{E_\fp}
\\
\ED rpq\col \EH rpq
  &\lra \EH r{p-r}{q+r-1}
  \end{align*}
one has $\EH 2pq=0$ for $q>s$, and also for $p>-d$ because each $\HH
q{E_\fp}$ is a finite direct sum of copies of $S_{\fp}$. Therefore, the sequence
converges strongly and yields
   \[
\Ext {i}{R_\fp}k{E_\fp} \cong
  \begin{cases}
  0 &\text{for }i< d-s\,,\\
  \Ext d{R_\fp}k{\HH s{E_\fp}}\ne 0&\text{for } i= d-s\,.
  \end{cases}
   \]
The formula above implies $\depth_{R_\fp}\!{E_\fp} = d-s$.  This gives
the first equality below:
   \begin{align*}
\depth{R_\fp} - \depth_{R_\fp}\!S_\fp
&= \depth{R_\fp} - \depth_{R_\fp}\!{E_\fp} - s \\
&= \sup\hh{k\otimes_{R_\fp}{E_\fp}} - s \\
& \le g-s \\
& \le g\,.
   \end{align*}
The second equality comes from \ref{abe}.
  \end{proof}

A complex in $\dcatN R$ is said to have \emph{finite injective dimension}
if it is isomorphic in $\dcat[R]$ to a bounded complex of injective
$R$-modules.  The next result, due to Ischebeck \cite[2.6]{Is}
when $M$ and $N$ are modules, can be deduced from~\cite[4.13]{CFF}.

\begin{lemma}
\label{ishebeck}
Let $R$ be a local ring and $N$ in $\dcatb R$ a complex of finite injective 
dimension. For each $M$ in $\dcatb R$ there is an equality
\[
\sup\{n\in\ZZ\mid \Ext nRMN\ne 0\} = \depth R - \depth M - \inf\hh N\,.
\]
\end{lemma}

\begin{proof}
Let $k$ be the residue field $k$ of $R$.  The first isomorphism below
holds because $N$ has finite injective dimension 
and $M$ is in $\dcatb R$, see \cite[4.4.I]{AF:hd}:
\begin{align*}
\hh{k\dtensor R\Rhom RMN} 
 &\cong \hh{\Rhom R{\Rhom RkM}N} \\
 &\cong \hh{\Rhom k{\Rhom RkM}{\Rhom RkN}} \\
 &\cong \Hom k{\hh{\Rhom RkM}}{\hh{\Rhom RkN}}\,.
\end{align*}
The other isomorphisms are standard. One deduces the 
second equality below:
\begin{align*}
\inf \hh{\Rhom RMN} &= \inf\hh{k\dtensor R\Rhom RMN}\\
 &= \inf \hh{\Rhom RkN} + \depth_{R}M\,.
\end{align*}
The first one comes from Lemma~\ref{lem:poincare}. In particular, 
for $M=R$ this yields
\[
\inf\hh{\Rhom RkN}=\inf\hh N - \depth R\,.
\]
Combining the preceding equalities, one obtains the desired assertion.
\end{proof}

The next result is due to Gabber~\cite[3.1.5]{Co}; Goto \cite{Go} had 
proved it for $N=R$.

\begin{proposition}
\label{prop:gabber}
For each $N$ in $\dcatb R$ the following conditions are equivalent.
\begin{enumerate}[\quad\rm(i)]
\item 
For each $\fp\in\Spec R$ the complex $N_{\fp}$ has finite injective 
dimension over $R_{\fp}$.
\item 
For each $M$ in $\dcatb R$ one has $\Ext nRMN=0$ for $n\gg0$.
\item[\quad\rm(ii$'$)] 
For each $\fm\in\Max R$ one has $\Ext nR{R/\fm}N=0$ for $n\gg0$.
\end{enumerate}
\end{proposition}

\begin{proof}
(i)$\implies$(ii). For each prime $\fp$, Lemma~\ref{ishebeck} yields 
the second equality below:
\begin{align*}
 - \inf\hh{{\Rhom RMN}_{\fp}} 
 &= - \inf{\hh{\Rhom {R_{\fp}} {M_{\fp}} {N_{\fp}} }}\\
 &=\depth R_{\fp} - \depth_{R_{\fp}}M_{\fp} - \inf\hh{N_{\fp}} \\
 &\leq \rfd RM - \inf\hh N \,.
\end{align*}
Theorem~\ref{thm:rfd} thus implies the desired result.

(ii$'$)$\implies$(i). Since $N$ is in $\dcatb R$ for each integer $n$ one has an isomorphism
\[
\Ext n{R_{\fm}}{R_{\fm}/\fm R_{\fm}}{N_{\fm}} \cong {\Ext nR{R/\fm}N}_{\fm}\,.
\]
Thus, the hypothesis and \ref{Bass series} imply $N_{\fm}$ has finite injective 
dimension over $R_{\fm}$.  By localization, $N_{\fp}$ has finite injective 
dimension over $R_{\fp}$ for each prime $\fp\subseteq\fm$.
  \end{proof}

 \begin{notes}
   \label{notes:rfd}
In \cite[2.1]{CFF} the number $\rfd RM$ is defined by the formula
 \[
\rfd RM=\sup\{n\in\BZ\mid\Tor nRTM\ne0\}\,,
  \]
where $T$ ranges over the $R$-modules of finite flat  dimension, and is called 
the \emph{large restricted flat dimension} of $M$ (whence, the notation).  We 
took as definition formula \eqref{eq:rfd}, which is due to Foxby 
(see~\cite[Notes, p.~131]{Ch1}) and is proved in \cite[5.3.6]{Ch1} and 
\cite[2.4(b)]{CFF}.  For $M$ of finite flat 
dimension one has $\rfd M=\fd_RM$, see \cite[5.4.2(b)]{Ch1} or \cite[2.5]{CFF}, 
and then \eqref{eq:rfd} goes back to Chouinard~\cite[1.2]{Cho}.  
  \end{notes}
  
\section{Derived reflexivity}
 \label{Derived reflexivity}

For every pair $C,M$ in $\dcat[R]$ there is a canonical \emph{biduality
morphism}
 \begin{equation}
   \label{eq:biduality}
\delta^C_M\col M\to \Rhom R{\Rhom RMC}C\,,
 \end{equation}
induced by the morphism of complexes
$m\mapsto(\alpha\mapsto(-1)^{|m||\alpha|}\alpha(m))$.  We say that
$M$ is \emph{derived $C$-reflexive} if both $M$ and $\Rhom RMC$ 
are in $\dcatb R$, and $\delta^C_M$ is an isomorphism.
Some authors write `$C$-reflexive' instead of `derived $C$-reflexive'.

Recall that the \emph{support} of a complex $M$ in $\dcatb R$ is the set
 \[
\Supp_RM=\{\fp\in\Spec R\mid \hh M_\fp\ne0 \}\,.
 \]

\begin{theorem}
\label{thm:refl}
Let $R$ be a noetherian ring and $C$ a complex in $\dcatb R$.

For each complex $M$ in $\dcatb R$ the following conditions are equivalent.
\begin{enumerate}[\quad\rm(i)]
  \item
$M$ is derived $C$-reflexive.
  \item
$\Rhom RMC$ is derived $C$-reflexive and $\Supp_{R}M\subseteq\Supp_{R}C$
holds.
  \item
$\Rhom RMC$ is in $\dcatP R$, and for every $\fm\in\Max R$ one has
\[
M_{\fm}\simeq \Rhom{R_{\fm}}{\Rhom{R_{\fm}}{M_{\fm}}{C_{\fm}}}{C_{\fm}}
\quad\text{in}\quad \dcat[R_{\fm}]\,.
\]
  \item
$U^{-1}M$ is derived $U^{-1}C$-reflexive for each multiplicatively closed set $U\subseteq R$.
  \end{enumerate}
\end{theorem}

The proof is based on a useful criterion for derived $C$-reflexivity.

\begin{chunk}
\label{triality}
Let $C$ and $M$ be complexes of $R$-modules, and set $\sh=\Rhom R-C$.

The composition $\sh(\delta_M^{C})\circ\delta^{C}_{\sh(M)}$ is the identity
map of $\sh(M)$ so the map
\[
\hh {\delta^{C}_{\sh(M)}}\col\hh {\sh(M)}\to\hh {\sh^3(M)}
\]
is a split monomorphism. Thus, if $\sh(M)$ is in $\dcatfg R$ and  
there exists
some isomorphism $\hh{\sh(M)}\cong \hh{\sh^{3}(M)}$, then 
$\delta^{C}_{\sh(M)}$ and $\sh(\delta_{M}^{C})$ are isomorphisms 
in $\dcat[R]$.
  \end{chunk}

The following proposition is an unpublished result of Foxby.

\begin{proposition}
\label{prop:foxby}
If for $C$ and $M$ in $\dcatb R$ there exists an isomorphism
  \[
\mu\col M \simeq \Rhom R{\Rhom RMC}C
 \quad\text{in}\quad \dcat[R]\,,
  \]
then the biduality morphism $\delta^{C}_{M}$ is an isomorphism as well.
  \end{proposition}

\begin{proof}
Set $\sh=\Rhom R-C$. Note that $\sh(M)$ is in $\dcatn R$ because $C$ and
$M$ are in $\dcatb R$. The morphism $\mu$ induces an isomorphism
$\hh{\sh^{3}(M)}\cong\hh{\sh(M)}$.  Each $R$-module
$\HH n{\sh(M)}=\Ext{-n}RMC$ is finite, so we conclude from \ref{triality} that
$\delta^{C}_{\sh(M)}$ is an isomorphism in $\dcat[R]$, hence
$\delta^{C}_{\sh^{2}(M)}$ is one as well.  The square
\[
\xymatrixrowsep{2.5pc}
\xymatrixcolsep{3pc}
\xymatrix{ M\ar@{->}[r]^{\mu}_{\simeq}
       \ar@{->}[d]_{\delta^{C}_M}
       & \sh^{2}(M) \ar@{->}[d]_{\simeq}^{\delta^{C}_{\sh^{2}(M)}}\\
       \sh^2(M)\ar@{->}[r]^{\sh^2(\mu)}_{\simeq} & \sh^{4}(M)}
\]
in $\dcat[R]$ commutes and implies that $\delta^{C}_M$ is an  
isomorphism,
as desired.
  \end{proof}

\begin{proof}[Proof of Theorem~\emph{\ref{thm:refl}}]
(i)$\implies$(ii).  This follows from \ref{triality} and \ref{support}.

(ii)$\implies$(i). Set $\sh=\Rhom R-C$ and form the exact triangle
in $\dcat[R]$:
\[
M\xra{\ \delta^{C}_{M}\ } \sh^{2}(M)\lra N \lra
\]
As $\sh(M)$ is $C$-reflexive, one has $\sh^{2}(M)\in\dcatb R$,
so the exact triangle above implies that $N$ is in $\dcatb R$. Since
$\Supp_{R}M\subseteq \Supp_{R}C$ holds, using \ref{support} one
obtains
  \begin{align*}
\Supp_{R}N &\subseteq \Supp_{R} M\cup\Supp_{R}\sh^{2}(M)\\
&=\Supp_{R} M\cup(\Supp_{R}M\cap\Supp_RC) \subseteq \Supp_{R}C\,.
  \end{align*}
On the other hand, the exact triangle above induces an exact triangle
\[
\sh(N) \lra \sh^{3}(M)\xra{\ \sh(\delta^{C}_{M})\ } \sh(M) \lra
\]
Since $\sh(M)$ is $C$-reflexive $\delta^{C}_{\sh(M)}$ is an isomorphism, so \ref{triality} shows that $\sh(\delta_M^{C})$ is an isomorphism as well. The second exact triangle now  gives $\hh{\sh(N)}=0$. The
already established inclusion $\Supp_{R}N\subseteq \Supp_{R}C$ and \ref{support} yield
\[
\Supp_RN=\Supp_{R}N\cap\Supp_{R}C =\Supp_{R}\Rhom RNC=\varnothing\,.
\]
This implies $N=0$ in $\dcat[R]$, and hence $\delta^{C}_{M}$ is an isomorphism.

(i)$\implies$(iv).
This is a consequence of the hypothesis $\Rhom RMC\in\dcatp R$.

(iv)$\implies$(iii).  With $U=\{1\}$ the hypotheses in (iv) implies $\Rhom RMC$ is in $\dcatp R$, while the isomorphism in (iii) is the special case $U=R\setminus \fm$.

(iii)$\implies$(i). For each $\fm\in\Max R$, Proposition~\ref{prop:foxby} yields that $\delta^{C_\fm}_{M_\fm}$ is an isomorphism, in $\dcat[R_{\fm}]$.  One has a  canonical isomorphism
  \[
\lambda_\fm\col\Rhom R{\Rhom RMC}C_\fm\xra{\,\simeq\,}
\Rhom{R_\fm}{\Rhom{R_\fm}{M_\fm}{C_\fm}}{C_\fm}
  \]
because $\Rhom RMC$ is in $\dcatp R$ and $M$ is in $\dcatb R$. Now using the equality $\delta^{C_\fm}_{M_\fm}=\lambda_\fm(\delta^C_M)_\fm$ one sees that $(\delta^C_M)_\fm$ is an isomorphism, and hence so is $\delta^C_M$.
  \end{proof}

\section{Semidualizing complexes}
\label{Semidualizing complexes}

For each complex $C$ there is a canonical \emph{homothety morphism}
 \begin{equation}
   \label{eq:homothety}
\chi^C\col R\to\Rhom RCC\quad\text{in}\quad \dcat[R]
 \end{equation}
induced by $r\mapsto(c\mapsto rc)$.
As in \cite[2.1]{Ch2}, we say that $C$ is \emph{semidualizing} if it  
is in
$\dcatb R$ and $\chi^C$ an isomorphism.  We bundle convenient
recognition criteria in:

 \begin{proposition}
  \label{prop:semidualizing}
For a complex $C$ in $\dcatb R$ the following are equivalent:
  \begin{enumerate}[\rm\quad(i)]
 \item
$C$ is semidualizing.
 \item[\rm(i$'$)]
$R$ is derived $C$-reflexive.
 \item
$C$ is derived $C$-reflexive and $\Supp_{R}C=\Spec R$.
 \item
For each $\fm\in\Max R$ there is an isomorphism
  \[
R_{\fm}\simeq \Rhom{R_\fm}{C_\fm}{C_\fm}
\quad\text{in}\quad \dcat[R_{\fm}]\,.
  \]
  \item
$U^{-1}C$ is semidualizing for $U^{-1}R$ for each multiplicatively  
closed set $U\subseteq R.$
   \end{enumerate}
 \end{proposition}

 \begin{proof}
To see that (i) and (i$'$) are equivalent, decompose $\chi^C$ as
  \[
R\xra{\,\delta^C_R\,}\Rhom R{\Rhom RRC}C
\xra{\,\simeq\,}\Rhom RCC
  \]
with isomorphism induced by the canonical isomorphism
$C\xra{\simeq}\Rhom RRC$.  Conditions (i$'$) through (iv) are equivalent
by Theorem \ref{thm:refl} applied with $M=R$.
 \end{proof}

Next we establish a remarkable property of semidualizing complexes.   
It uses the invariant $\rfd R{(-)}$ discussed in Section~\ref{Depth}.

\begin{theorem}
\label{semid:dualfinite}
If $C$ is a semidualizing complex for $R$ and $L$ is a complex in $ 
\dcatn R$
with $\Rhom RLC\in\dcatb R$, then $L$ is in $\dcatb R$; more  
precisely, one
has
  \begin{align*}
\inf \hh L&\geq\inf \hh C - \rfd R{\Rhom RLC}>-\infty\,.
  \end{align*}
\end{theorem}

\begin{proof}
For each $\fm\in\Max R\cap\Supp_RL$ one has a chain of relations
  \begin{align*}
\inf\hh{L_{\fm}}
&=  - \depth_{R_{\fm}}C_{\fm} + \depth_{R_{\fm}}{\Rhom RLC}_{\fm}\\
& = \inf \hh{C_{\fm}} - \depth R_{\fm} + \depth_{R_{\fm}}{\Rhom RLC}_ 
{\fm}\\
&\geq \inf \hh C - \rfd R{\Rhom RLC}\\
&>-\infty
  \end{align*}
with equalities given by Lemma \ref{lem:bass}, applied first with $M=L$ and $N=C$, then with $M=C=N$; the first inequality is clear, and the second one holds by Theorem~\ref{thm:rfd}.  Now use the equality $\inf\hh L=\inf_{\fm\in\Max R}\{\inf\hh{L_\fm}\}$.
  \end{proof}

The next theorem parallels~Theorem~\ref{thm:refl}. The impact of the hypothesis that $C$ is semidualizing can be seen by comparing condition (iii) in these results: one need not assume $\Rhom RMC$ is bounded.  In particular, reflexivity with respect to a semidualizing complex can now be \emph{defined} by means of property (i$'$) alone. Antecedents of the theorem are discussed in \ref{notes:semid}.

   \begin{theorem}
     \label{thm:bounded-reflexive}
Let $C$ be a semidualizing complex for $R$.

For a complex $M$ in $\dcatb R$ the following conditions are equivalent:
\begin{enumerate}[\quad\rm(i)]
  \item
$M$ is derived $C$-reflexive.
  \item[\rm(i$'$)]
There exists an isomorphism $M\simeq\Rhom R{\Rhom RMC}C$.
  \item
$\Rhom RMC$ is derived $C$-reflexive.
   \item
For each $\fm\in\Max R$ there is an isomorphism
  \[
M_{\fm}\simeq \Rhom{R_{\fm}}{\Rhom{R_{\fm}}{M_{\fm}}{C_{\fm}}}{C_{\fm}}
\quad\text{in}\quad \dcat[R_{\fm}]\,.
  \]
  \end{enumerate}
Furthermore, these conditions imply the following inequalities
  \[
\amp\hh{\Rhom RMC}\le\amp \hh C-\inf \hh M+\rfd RM <\infty\,.
  \]
 \end{theorem}

\begin{proof}
(i)$\iff$(ii). Apply Theorem~\ref{thm:refl}, noting that $\Supp_RC=\Spec R$, by \ref{support}.

(i)$\implies$(i$'$).  This implication is a tautology.

(i$'$)$\implies$(iii).  This holds because Theorem~\ref{semid:dualfinite},
applied with $L=\Rhom RMC$, shows that $\Rhom RMC$ is bounded, and
so the given isomorphism localizes.

(iii)$\implies$(i).
For each $\fm\in\Max R$ the complex $C_\fm$ is semidualizing for $R_\fm$ by Proposition \ref{prop:semidualizing}.  One then has a chain of (in)equalities
  \begin{align*}
\inf\hh{\Rhom RMC}
&=\inf_{\fm\in\Max R}\{\inf\hh{\Rhom RMC_\fm}\}
 \\
&=\inf_{\fm\in\Max R}\{\inf\hh{\Rhom{R_\fm}{M_\fm}{C_\fm}}\}
 \\
&\ge\inf_{\fm\in\Max R}\{\inf\hh{C_{\fm}}-\rfd{R_{\fm}}{M_\fm}\}
 \\
&\ge\inf\hh{C}-\sup_{\fm\in\Max R}\{\rfd{R_{\fm}}{M_\fm}\} 
 \\
&=\inf\hh{C}-\rfd{R}{M}\\
&>-\infty\,,
  \end{align*}
where the first inequality comes from Theorem \ref{semid:dualfinite} applied over $R_\fm$ to the complex $L=\Rhom{R_\fm}{M_\fm}{C_\fm}$, while the last inequality is given by Theorem \ref{thm:rfd}.  
It now follows from Theorem~\ref{thm:refl} that $M$ is derived $C$-reflexive.

The relations above and \ref{derived-functors} yield the desired bounds on amplitude.
  \end{proof}

 \begin{notes}
   \label{notes:semid}
The equivalence (i)$\iff$(ii) in Theorem \ref{thm:bounded-reflexive}
follows from \cite[2.1.10]{Ch1} and \cite[2.11]{Ch2}; see \cite[3.3]{FST}.
When $\dim R$ is finite, a weaker form of (iii)$\implies$(i) is proved in
\cite[2.8]{FS}: $\delta^{C_\fm}_{M_\fm}$ an isomorphism for all 
$\fm\in\Max R$ implies that $\delta^C_M$ is one.

When $R$ is Cohen-Macaulay and local each semidualizing complex $C$ satisfies 
$\amp\hh C=0$, so it is isomorphic to a shift of a finite module; see~\cite[3.4]{Ch2}.
 \end{notes}

\section{Perfect complexes}
\label{Perfect complexes}

Recall that a complex of $R$-modules is said to be \emph{perfect} if it is isomorphic in $\dcat[R]$ to a bounded complex of finite projective modules.  For ease of reference we collect, with complete proofs, some useful tests for perfection; the equivalence of (i) and (ii) is contained in \cite[2.1.10]{Ch1}, while the argument that (i) are (iii) are equivalent is modelled on a proof when $M$ is a module, due to Bass and Murthy~\cite[4.5]{BM}.

\begin{theorem}
\label{thm:perfect}
For a complex $M$ in $\dcatb R$ the following conditions are equivalent.
\begin{enumerate}[\quad\rm(i)]
  \item
$M$ is perfect.
  \item
$\Rhom RMR$ is perfect.
  \item
$M_\fm$ is perfect in $\dcat[R_\fm]$ for each $\fm\in\Max R$.
  \item[\rm(iii$'$)]
$P^{R_\fm}_{M_\fm}(t)$ is a Laurent polynomial for each $\fm\in\Max R$.
  \item
$U^{-1}M$ is perfect in $\dcat[U^{-1}R]$ for each multiplicatively
closed set $U\subseteq R$.
  \end{enumerate}
    \end{theorem}

 \begin{proof}
(iv)$\implies$(iii).
This implication is a tautology.

(iii)$\implies$(i).  Choose a resolution $F\tiso M$ with each $F^{i}$ finite
free and zero for $i\ll 0$. Set $s=\sup\hh F+1$ and $H=\image(\dd^F_{s})$, and
note that the complex $\Shift^{-s}F_{\ges s}$ is a free resolution
of $H$. Since each $R$-module $\image(\dd^F_n)$ is finite, the subset
of primes $\fp\in\Spec R$ with $\image(\dd^F_n)_{\fp}$ projective over
$R_{\fp}$ is open.  It follows that the set
 \[
D_{n}= \{\fp\in \Spec R\mid \pd_{R_{\fp}}H_\fp\le n\}
 \]
is open in $\Spec R$ for every $n\ge 0$.  One has $D_n\subseteq D_{n+1}$
for $n\ge0$, and the hypothesis means $\bigcup_{n\ges0}D_{n}=\Spec R$.
As $\Spec R$ is noetherian, it follows that $D_p=\Spec R$ holds for
some $p\ge0$, so that $\image(\dd^F_{s+p})$ is projective. Taking  
$E_n=0$
for $n>s+p$, $E_{s+p}=\image(\dd^F_{s+p})$, and $E_n=F_n$ for $n<s+p$
one gets a perfect subcomplex $E$ of $F$.  The inclusion $E\to F$ is a
quasi-isomorphism, so $F$ is perfect.

(i)$\implies$(iv) and (i)$\implies$(ii).  In $\dcat[R]$ one has $M\simeq F$ with
$F$ a bounded complex of finite projective $R$-modules.  This implies
isomorphisms $U^{-1}M\simeq U^{-1}F$ in $\dcat[U^{-1}R]$ and 
$\Rhom RMR\simeq\Rhom RFR$ in $\dcat[R]$, with  bounded 
complexes of finite projective modules on their right hand sides.

(ii)$\implies$(i).  The perfect complex $N=\Rhom RMR$ is evidently derived 
$R$-reflexive, so the implication (ii)$\implies$(i) in Theorem~\ref{thm:refl} 
applied with $C=R$ gives $M\simeq\Rhom RNR$; as we have just seen, 
$\Rhom RNR$ is perfect along with $N$.

(iii)$\iff$(iii$'$).
We may assume that $R$ is local with maximal ideal $\fm$.  By \ref{minimal}, 
there is an isomorphism $F\simeq M$ in $\dcat[R]$, with each $F_{n}$ finite 
free, $\dd(F)\subseteq \fm F$, and $P^R_M(t)=\sum_{n\in\BZ}\rank_RF_nt^n$.  
Thus, $M$ is perfect if and  only if $F_{n}=0$ holds for all $n\gg 0$; that is, 
if and only if $P^R_M(t)$ is a Laurent polynomial.
  \end{proof}

The following elementary property of perfect complexes is well known:

\begin{chunk}
\label{ch:perfect-tensor}
If $M$ and $N$ are perfect complexes, then so are $M\dtensor RN$ and 
$\Rhom RMN$.
 \end{chunk}

To prove a converse we use a version of a result from \cite{FI},
which incorporates a deep result in commutative algebra, namely, the New 
Intersection Theorem.

\begin{theorem}
\label{thm:perfect-bounds}
When $M$ is a perfect complex of $R$-modules and $N$ a complex in $\dcatfg R$
satisfying $\Supp_{R}N\subseteq \Supp_{R}M$, the following inequalities hold:
\begin{align*}
\sup\hh N  &\leq \sup\hh{M\dtensor RN} - \inf \hh M \\
\inf\hh N  &\geq \inf \hh{M\dtensor RN} -  \sup \hh M \\
\amp \hh N & \leq \amp \hh{M\dtensor RN} + \amp \hh M
\end{align*}
If $M\dtensor RN$ or $\Rhom RMN$ is in $\dcatB R$, then $N$ is in $\dcatb R$.
  \end{theorem}

\begin{proof}
For each $\fp$ in $\Supp_{R}N$ the complex $M_{\fp}$ is perfect and non-zero 
in $\dcat[R_{\fp}]$.

The second link in the following chain comes from \cite[3.1]{FI}, the rest 
are standard:
\begin{align*}
\sup{\hh N}_{\fp} 
                  &=\sup{\hh{N_{\fp}}} \\
                  &\le \sup\hh{M_{\fp} \dtensor{R_{\fp}}{N_{\fp}}} - \inf{\hh{M_{\fp}}}\\
                  &= \sup \hh{M\dtensor RN}_{\fp} -  \inf{\hh M}_{\fp} \\
                  &\leq \sup\hh{M\dtensor RN} - \inf {\hh M}
\end{align*}
The first inequality follows, as one has 
$\sup \hh N=\sup_{\fp\in\Supp N}\{\sup{\hh N}_{\fp}\}$.
 
Lemma \ref{lem:poincare} gives the second link in the next chain,
the rest are standard:
  \begin{align*}
\inf{\hh N}_{\fp} 
                  &=\inf{\hh{N_{\fp}}} \\
                  &= \inf\hh{M_{\fp} \dtensor{R_{\fp}}{N_{\fp}}} - \inf{{\hh M}_{\fp}} \\
                  &= \inf \hh{M\dtensor RN}_{\fp} - \inf{\hh M}_{\fp} \\
                  &\geq \inf\hh{M\dtensor RN} - \sup {\hh M}
\end{align*}
The second inequality follows, as one has $\inf{\hh N}=\inf_{\fp\in\Supp N}\{\inf{\hh N}_{\fp}\}$.

The first two inequalities imply the third one, which contains
the assertion concerning $M\dtensor RN$.  In turn, it implies 
the assertion concerning $\Rhom RMN$, because the complex
$\Rhom RMR$ is  perfect along with $M$, one has
  \[
\Supp_RN\subseteq\Supp_{R}M=\Supp_{R}\Rhom RMR
  \]
due  to  \ref{support}, and there is a canonical isomorphism
\[
\Rhom RMR\dtensor{R}N\simeq\Rhom RMN\,.
 \qedhere
\]
  \end{proof}

\begin{corollary}
\label{cor:perfect-tensor}
Let $M$ be a perfect complex and $N$ a complex in $\dcatfg R$ satisfying $\Supp_{R}N\subseteq \Supp_{R}M$. If $M\dtensor RN$ or $\Rhom RMN$ is perfect, then so is $N$.
 \end{corollary}

\begin{proof}
Suppose $M\dtensor RN$ is perfect; then  $N\in\dcatb R$ holds, by Theorem~\ref{thm:perfect-bounds}. For each $\fm\in\Max R$, Theorem~\ref{thm:perfect} and  Lemma~\ref{lem:poincare} imply that $P^{R_\fm}_{M_\fm}(t)P^{R_\fm}_{N_\fm}(t)$ is a Laurent polynomial, and hence so is $P^{R_{\fm}}_{N_{\fm}}(t)$.  Another  application of Theorem~\ref{thm:perfect} now shows that $N$ is perfect.

The statement about $\Rhom RMN$ follows from the one concerning derived tensor products, by using the argument for the last assertion of the theorem.
  \end{proof}

Next we establish a stability property of derived  reflexivity. The forward implication is well known; see, for instance, \cite[3.17]{Ch2}.

\begin{theorem}
\label{thm:perfect-reflexive}
Let $M$ be a perfect complex and $C$ a complex in $\dcatn R$.

If $N$ in $\dcat[R]$ is derived $C$-reflexive, then so is $M\dtensor RN$.

Conversely, for $N$ in $\dcatfg R$ satisfying $\Supp_{R}N\subseteq\Supp_{R}M$,
if $M\dtensor RN$ is derived $C$-reflexive, then so is $N$.
 \end{theorem}

\begin{proof}
We may  assume that $M$ is a bounded complex of finite projective 
$R$-modules.  

Note that derived $C$-reflexivity is preserved by translation, 
direct sums, and direct summands, and that if two of the complexes in 
some exact triangle are derived $C$-reflexive, then so is the third. A 
standard induction on the number of non-zero components of $M$ 
shows that when $N$ is derived $C$-reflexive, so is $M\otimes_{R}N$.

Assume that $M\dtensor RN$ is derived $C$-reflexive and $\Supp_{R}N\subseteq\Supp_{R}M$ holds.
Theorem~\ref{thm:perfect-bounds} gives $N\in\dcatb R$. For the complex  $M^{*}=\Rhom RMR$ and the functor $\sh(-)=\Rhom R-C$, in $\dcat[R]$ there is a natural isomorphism
\[
M^{*}\dtensor R\sh(N)\simeq\sh(M\dtensor RN)\,.
  \]
Now $\sh(N)$ is in $\dcatfg R$ because $N$ is in $\dcatb R$ and $C$ is in $\dcatn R$, by \cite[p.\,92, 3.3]{H}. Since $M$ is perfect, one has that
  \[
\Supp_{R}\sh(N)\subseteq\Supp_{R}N\subseteq\Supp_{R}M=\Supp_{R}M^{*},
  \]
so Theorem~\ref{thm:perfect-bounds} gives $\sh(N)\in\dcatb R$.
Thus, $\sh^{2}(N)$ is in $\dcatfg R$, so the isomorphism
\begin{equation}
  \label{eq:hsquare}
M\dtensor R\sh^{2}(N)\simeq\sh^{2}(M\dtensor RN)
\end{equation}
and Theorem~\ref{thm:perfect-bounds} yield $\sh^2(N)\in\dcatb R$. 
Forming an exact triangle
\[
N\xra{\ \delta^{C}_{N}\ } \sh^{2}(N)\lra W \lra
\]
one then gets $W\in\dcatb R$ and $\Supp_{R}W\subseteq \Supp_{R}N$.

In the induced exact triangle
\[
M\dtensor RN\xra{\ M\dtensor R\delta^{C}_{N}\ } M\dtensor R\sh^{2}(N)
\lra M\dtensor RW \lra
\]
the morphism $M\dtensor{R}\delta^{C}_{N}$ is an isomorphism, as
its composition with the isomorphism in \eqref{eq:hsquare} is equal to 
$\delta^{C}_{M\dtensor RN}$, which is an isomorphism by hypothesis. 
Thus, we obtain $M\dtensor RW=0$ in $\dcat[R]$, hence $W=0$ by 
\ref{support}, so $\delta^{C}_{N}$ is a isomorphism.
  \end{proof}

Sometimes, the perfection of a complex can be deduced from its homology.

Let $H$ be a graded $R$-module.  We say that $H$ is (\emph{finite})
\emph{graded projective} if it is bounded and for each $i\in\BZ$ the 
$R$-module $H_i$ is  (finite) projective.  

  \begin{chunk}
    \label{ch:projhom}
If $M$ is a complex of $R$-modules such that $\hh M$ is projective, then $M\simeq\hh M$ in $\dcat[R]$, by \cite[1.6]{AILN}. Thus when $\hh M$ is in addition finite, $M$ is perfect.
  \end{chunk}

We recall some facts about projectivity and idempotents; see also \cite[2.5]{AI:mm}.

  \begin{chunk}
    \label{ch:projrank}
Let $H$ be a finite graded projective $R$-module.  

The $R_\fp$-module $(H_i)_\fp$ then is finite free for every $\fp\in\Spec R$ 
and every $i\in\BZ$, and one has $(H_i)_\fp=0$ for almost all $i$, so $H$ 
defines a function
  \[
r_{H}\col \Spec R\to\BN
  \quad\text{given by}\quad
r_{H}(\fp)=\sum_{i\in\BZ}\rank_{R_\fp}(H_i)_\fp\,.
  \]
One has $r_{H}(\fp)=\rank_{R_\fp}\big(\bigoplus_{i\in\BZ}H_i\big)_{\fp}$; since the $R$-module $\bigoplus_{i\in\BZ}H_i$ is finite projective, $r_{H}$ is constant on each connected component of $\Spec R$.

We say that $H$ \emph{has rank $d$}, and write $\rank_RH=d$, if  $r_{H}(\fp)=d$ holds for every $\fp\in\Spec R$.  We say that $H$ is \emph{invertible} if it is graded projective of rank $1$.
 \end{chunk}

\begin{chunk}
\label{ch:idempotents}
Let $\{a_1,\dots,a_s\}$ be the (unique) complete set of orthogonal primitive idempotents of $R$.  The open subsets $D_{a_i}=\{\fp\in\Spec R\mid\fp\not\ni a_i\}$ for $i=1,\dots,s$ are then the distinct connected components of $\Spec R$.

An element $a$ of $R$ is idempotent if and only if  $a=a_{i_1}+\cdots+a_{i_r}$ with indices $1\le i_1<\cdots<i_r\le s$; this sequence of indices is uniquely determined.  

Let $a$ be an idempotent and $-_a$ denote localization at the multiplicatively closed set $\{1,a\}$ of $R$.  For all $M$ and $N$ in $\dcat[R]$ there are canonical isomorphisms
 \begin{gather*}
M\simeq M_a\oplus M_{1-a}
 \quad\text{and}\\
\Rhom R{M_a}N\simeq\Rhom R{M_a}{N_a}\simeq\Rhom RM{N_a}\>\>.
 \end{gather*}
In particular, when $M$ is in $\dcatb R$ so is $M_a$, and there is an isomorphism 
$M\simeq M_a$ in $\dcat[R]$ if and only if one has $\Supp_RM=D_a$.  

Every graded $R$-module $L$ has a canonical decomposition $L=\bigoplus_{i=1}^sL_{a_i}$.  
  \end{chunk}

The next result sounds---for the first time in this paper---the theme of rigidity.

\begin{theorem}
\label{thm:invertible}
Let $L$ be a complex in $\dcatn R$.

If $M$ in $\dcatb R$ satisfies $\Supp_{R}M\supseteq \Supp_{R}L$ and
there is an isomorphism
  \[
M\simeq \Rhom RLM \quad\text{or}\quad M\simeq L\dtensor{R}M\,,
  \]
then for some idempotent $a$ in $R$ the $R_{a}$-module $\HH0L_a$ 
is invertible and one has
  \[
L\simeq\HH0L\simeq\HH0L_a\simeq L_a\quad\text{in}\quad\dcat[R]\,.
 \]
The element $a$ is determined by either one of the following equalities:
 \[
\Supp_{R}M =\{\fp\in\Spec R\mid \fp\not\ni a\} = \Supp_{R}L \,.
 \]
   \end{theorem}

\begin{proof}
If $\hh M=0$, then the hypotheses imply $\Supp_RL=\varnothing$, so $a=0$ is the desired idempotent.  For the rest of the proof we assume $\hh M\ne0$.

If $M\simeq\Rhom RLM$ holds and $\fm$ is in $\Max R\cap\Supp_RM$, then  
Lemma~\ref{lem:bass} shows that 
$L_{\fm}$ is in $\dcatp{R_{\fm}}$ and gives the second equality below:
\[
I_{R_\fm}^{M_\fm}(t)=I_{R_\fm}^{\Rhom RLM_\fm}(t)
=P^{R_\fm}_{L_\fm}(t)\cdot I_{R_\fm}^{M_\fm}(t)\,.
\]
As $I_{R_\fm}^{M_\fm}(t)\ne0$ by \ref{supp:bass}, this gives 
$P^{R_\fm}_{L_\fm}(t)=1$, and hence $L_\fm\simeq R_\fm$ by~\ref{minimal}.  
Thus, for every $\fp\in\Supp_{R}M$ one has $L_\fp\simeq R_\fp$, which yields
$\Supp_RM=\Supp_RL=\Supp_{R}\HH 0L$ and shows the $R$-module $\HH 0L$ 
is projective with $\rank_{R_\fp}{\HH 0 L}_\fp= 1$ for each $\fp\in\Supp_R\HH 0L$.  
The rank of a projective module is constant on connected components of $\Spec R$, 
therefore $\Supp_R\HH 0L$ is a union of such components, whence,
by \ref{ch:idempotents}, there is a unique idempotent $a\in R$, such that
   \[
\Supp_R\HH 0L=\{\fp\in\Spec R\mid \fp\not\ni a\},
   \]
and the graded $R_a$-module $\hh L_a$ is invertible.  The preceding 
discussion, \ref{ch:idempotents}, and \ref{ch:projhom} give isomorphisms 
$L\simeq\HH0L\simeq\HH0L_a\simeq L_a$ in $\dcat[R]$.

A similar argument, using Lemma~\ref{lem:poincare} and ~\ref{supp:poincare}, 
applies if $M\simeq L\dtensor RM$.
 \end{proof}

\section{Invertible complexes}
\label{Invertible complexes}

We say that a complex in $\dcat[R]$ is \emph{invertible} if it is semidualizing 
and perfect.

The following canonical morphisms, defined for all $L$, $M$, and $N$ in 
$\dcat[R]$, play a role in characterizing invertible complexes and in 
using them.  \emph{Evaluation}
  \begin{equation}
    \label{eq:evaluation}
\Rhom RLN\dtensor {R} L \lra N\,.
  \end{equation}
is induced by the chain map $\lambda\otimes l\mapsto\lambda(l)$.   
\emph{Tensor-evaluation} is the composition
  \begin{equation}
    \label{eq:tensor_evaluation}
 \begin{aligned}
\Rhom R{M\dtensor{R} L}{N}\dtensor{R}L 
 &\xra{\ \simeq\ } \Rhom R{L\dtensor{R}M}{N}\dtensor{R}L\\
 &\xra{\ \simeq\ } \Rhom RL{\Rhom R{M}{N}}\dtensor{R}L \\
 &\xra{\ {\phantom{\simeq}}\ }\Rhom R{M}{N}
\end{aligned}
  \end{equation}
where the isomorphisms are canonical and the last arrow is given by 
evaluation.

The equivalence of conditions (i) and (i$'$) in the result below shows
that for complexes with zero differential invertibility agrees with
the notion in~\ref{ch:projrank}. Invertible complexes coincide with
the \emph{tilting complexes} of Frankild, Sather-Wagstaff, and Taylor,
see \cite[4.7]{FST}, where some of the following equivalences are proved.

\begin{proposition}
\label{prop:semid}
For $L\in\dcatb R$ the following conditions are equivalent.
\begin{enumerate}[\quad\rm(i)]
  \item
$L$ is invertible in $\dcat[R]$.
  \item[\rm(i$'$)]
$\hh L$ is an invertible graded $R$-module.
  \item
$\Rhom RLR$ is invertible in $\dcat[R]$.
  \item[\rm(ii$'$)]
$\Ext{}RLR$ is an invertible graded $R$-module.
  \item
For each $\fp\in\Spec R$ one has $L_\fp\simeq\Shift^{r(\fp)}R_\fp$
in $\dcat[R_\fp]$ for some $r(\fp)\in\BZ$.
  \item[\rm(iii$'$)]
For each $\fm\in\Max R$ one has\,\ $P^{R_\fm}_{L_\fm}=t^{r(\fm)}$ for
some $r(\fm)\in\BZ$.
  \item
$U^{-1}L$ is invertible in $\dcat[U^{-1}R]$ for each multiplicatively
closed set $U\subseteq R.$
  \item
For some $N$ in $\dcatfg R$ there is an isomorphism $N\dtensor RL\simeq
R$.
  \item
For each $N$ in $\dcat[R]$ the evaluation map \eqref{eq:evaluation}
is an isomorphism.
 \item[\rm(vi$'$)]
For all $M$, $N$ in $\dcat[R]$ the tensor-evaluation map
  \eqref{eq:tensor_evaluation} is an isomorphism.
 \end{enumerate}
   \end{proposition}

\begin{proof}
(i)$\iff$(iv).
This follows from Proposition \ref{prop:semidualizing} and 
Theorem~\ref{thm:perfect}.

(i)$\implies$(vi). 
The first two isomorphisms below holds because $L$ is perfect:
\[
\Rhom RLN\dtensor R L \simeq \Rhom RL{L\dtensor RN} \simeq 
{\Rhom RLL}\dtensor RN \simeq N\,.
\]
The third one holds because $L$ is semidualizing.

(vi)$\implies$(vi$'$).  In \eqref{eq:tensor_evaluation}, use
\eqref{eq:evaluation} with $\Rhom RMN$ in place of $N$.

(vi$'$)$\implies$(vi). Set $M=R$ in \eqref{eq:tensor_evaluation}.

(vi)$\implies$(v). Setting $N=R$ one gets an isomorphism $\Rhom
RLR\dtensor RL\simeq R$. Note that $\Rhom RLR$ is in $\dcatn R$, 
since $L$ is in $\dcatb R$.

Condition (v) localizes, and the already proved equivalence of (i) and
(iv) shows that conditions (i) and (ii) can be checked locally.  Clearly,
the same holds true for conditions (i$'$), (ii$'$), (iii$'$), and (iii).
Thus, in order to finish the proof it suffices to show that when $R$
is a local ring there exists a string of implications linking (v) to
(i) and passing through the remaining conditions.

(v)$\implies$(iii$'$).
Lemma \ref{lem:poincare} gives $P^R_{N}(t)\cdot P^R_L(t)=1$.
Such an equality of formal Laurent series implies $P^R_L(t)=t^r$
and $P^R_{N}(t)=t^{-r}$ for some integer $r$.

(iii$'$)$\implies$(iii).
This follows from \ref{minimal}.

(iii)$\implies$(i$'$).
This implication is evident.

(i$'$)$\implies$(ii$'$).
As $\hh L$ is projective one has $L\simeq\hh L$ in $\dcat[R]$, see
\ref{ch:projhom}, hence
\[
\Ext{}RLR\cong \Ext{}R{\hh L}R\cong \Hom R{\hh L}R\,.
\]
Now note that the graded module $\Hom R{\hh L}R$ is invertible because
$\hh L$ is.

(ii$'$)$\implies$(ii).
Because $\hh{\Rhom RLR}$ is projective, \ref{ch:projhom} gives
the first isomorphism below; the second one holds (for some $r\in\BZ$)
because $R$ is local:
\[
\Rhom RLR\simeq\hh{\Rhom RLR}=\Ext{}RLR\simeq\Shift^rR\,.
\]

(ii)$\implies$(i).  The invertible complex $L'=\Rhom RLR$ is evidently 
derived $R$-reflexive, so the implication (ii)$\implies$(i) in 
Theorem~\ref{thm:refl} applies with $C=R$. It gives 
$L\simeq\Rhom R{L'}R$; now note that $\Rhom R{L'}R$ is invertible 
along with $L$.
 \end{proof}

Recall that $\pic R$ denotes the \emph{Picard group} of $R$, whose
elements are isomorphism classes of invertible $R$-modules, 
multiplication is  induced by tensor product over $R$, and the class
of $\Hom RLR$ is the inverse of that of $L$.  A derived version of this
construction is given in \cite[4.1]{FST} and is recalled below; it coincides
with the derived Picard group of $R$ relative to itself, in the sense of 
Yekutieli \cite[3.1]{Ye}.

 \begin{chunk}
\label{picard}
When $L$ is an invertible complex, we set 
  \[
L^{-1}=\Rhom RLR\,.
  \]
Condition (vi) of Proposition~\ref{prop:semid} gives for each 
$N\in\dcat[R]$ an isomorphism
  \[
\Rhom RLN\simeq L^{-1}\dtensor RN\,.
  \]

In view of~\ref{ch:projhom}, condition (i$'$) of
Proposition~\ref{prop:semid} implies that the isomorphism classes $[L]$
of invertible complexes $L$ in $\dcat[R]$ form a set, which we denote
$\dpic R$.
As derived tensor products are associative and commutative, $\dpic R$ 
carries a natural structure of abelian group, with unit element $[R]$, and
$[L]^{-1}=[L^{-1}]$; cf.~\cite[4.3.1]{FST}.  Following \emph{loc.~cit.},
we refer to it as the \emph{derived Picard group} of $R$.

We say that complexes $M$ and $N$ are \emph{derived Picard equivalent}
if there is an isomorphism $N\simeq L\dtensor RM$ for some invertible 
complex $L$.

Clearly, if $N$ and $N'$ are complexes in $\dcat[R]$ which satisfy
$L\dtensor R N\simeq L\dtensor R N'$ or $\Rhom RLN\simeq \Rhom RL{N'}$,
then $N\simeq N'$.
    \end{chunk}

The derived Picard group of a local ring $R$ is the free abelian group
with generator $[\Shift R]$; see \cite[4.3.4]{FST}.  In general, one
has the following description, which is a special case of \cite[3.5]{Ye}.
We include a proof, for the sake of completeness.

  \begin{proposition}
There exists a canonical isomorphism of abelian groups
 \[
\dpic R\xra{\,\cong\,}\prod_{i=1}^s\big(\pic{R_{a_i}}\times\BZ\big)\,,
 \]
where $\{a_1,\dots,a_s\}$ is the complete set of primitive orthogonal 
idempotents; see \emph{\ref{ch:idempotents}}.
  \end{proposition}

 \begin{proof}
By Proposition \ref{prop:semid}, every element of $\dpic R$ is equal
to $[L]$ for some graded invertible $R$-module $L$.  In the canonical
decomposition from \ref{ch:idempotents} each $R_{a_i}$-module 
$L_{a_i}$ is graded invertible.  It is indecomposable because 
$\Spec\big(R_{a_i}\big)$ is connected, hence $L_{a_i}\cong\Shift^{n_i}L_i$ 
with uniquely determined invertible $R_{a_i}$-module $L_i$ and $n_i\in\BZ$.  
The map $[L]\mapsto\big(([L_1],n_1),\dots,([L_s],n_s)\big)$ gives the desired
isomorphism.
 \end{proof}

Other useful properties of derived Picard group actions are collected
in the next two results, which overlap with \cite[4.8]{FST}; we include
proofs for completeness.

\begin{lemma}
\label{refl:operations}
For $L$ invertible, and $C$ and $M$ in $\dcatb R$, the following are
equivalent.
  \begin{enumerate}[\quad\rm(i)]
   \item
$M$ is derived $C$-reflexive.
  \item
$M$ is derived $L\dtensor RC$-reflexive.
   \item
$L\dtensor{R}M$ is derived $C$-reflexive.
\end{enumerate}
 \end{lemma}

\begin{proof}
(i)$\implies$(ii). Since $L$ is invertible, the morphism
\[
\vartheta\col L\dtensor R{\Rhom RMC}\to\Rhom RM{L\dtensor RC}
\]
represented by $l\otimes\alpha\mapsto(m\mapsto l\otimes\alpha(m))$, is an isomorphism: It suffices to check the assertion after localizing at each $\fp\in\Spec R$, where it follows from $L_\fp\cong  R_\fp$.
In particular, since $\Rhom RMC$ is in $\dcatb R$, so is $\Rhom RM{L\dtensor RC}$. Furthermore, in $\dcat[R]$ there is a commutative diagram of canonical morphisms
 \[
\xymatrixcolsep{3pc}
\xymatrixrowsep{2.5pc}
\xymatrix {
M\ar@{->}[r]^-{\delta^{L\dtensor RC}_M}\ar@{->}[d]_{\delta^C_M}^-{\simeq}
&\Rhom R{\Rhom RM{L\dtensor RC}}{L\dtensor RC}
    \ar@{->}[d]^{\Rhom R{\vartheta}{L\dtensor RC}}_-{\simeq}
  \\
\Rhom R{\Rhom RMC}C\ar@{->}[r]^-{\lambda}_-{\simeq}
&\Rhom R{L\dtensor R{\Rhom RMC}}{L\dtensor RC}
}
\]
with $\lambda(\alpha)=L\dtensor R\alpha$, which is an isomorphism, as is readily verified by localization.
Thus, $M$ is derived $L \dtensor RC$-reflexive.

(ii)$\implies$(i). The already established implication (i)$\implies$(ii) shows that $M$ is reflexive with respect to $L^{-1}\dtensor R(L\dtensor RC)$, which is isomorphic to $C$.

(i)$\iff$(iii)  This follows from Theorem \ref{thm:perfect-reflexive}.
  \end{proof}

{From} Proposition~\ref{prop:semidualizing} and Lemma~\ref{refl:operations}, we obtain:

\begin{lemma}
\label{refl:operations2}
For $L$ invertible and $C$ in $\dcatb R$ the following are equivalent.
  \begin{enumerate}[\quad\rm(i)]
   \item
$C$ is semidualizing.
 \item
$L\dtensor RC$ is semidualizing.
   \item
$L$ is derived $C$-reflexive.
  \qed
     \end{enumerate}
\end{lemma}

Invertible complexes are used in \cite[5.1]{FST} to characterize mutual
reflexivity of a pair of semidualizing complexes.  The next theorem is 
fundamentally different, in that the semidualizing property is part of its 
conclusions, not of its hypotheses.

\begin{theorem}
\label{thm:reflsemid}
For $B$ and $C$ in $\dcatb R$ the following conditions are equivalent.
\begin{enumerate}[\rm\quad(i)]
  \item
$B$ is derived $C$-reflexive, $C$ is derived $B$-reflexive, and
$\Supp_{R}B=\Spec R$.
  \item 
$B$ is semidualizing, $\Rhom RBC$ is invertible, and the evaluation map
$\Rhom RBC \dtensor RB\to C$ is an isomorphism in $\dcat[R]$.
 \item
$B$ and $C$ are semidualizing and derived Picard equivalent.
   \end{enumerate}
 \end{theorem}

\begin{proof}
(i)$\implies$(ii). The hypotheses pass to localizations and,  by
Propositions~\ref{prop:semidualizing} and \ref{prop:semid}, the
conclusions can be tested locally. We may thus assume $R$ is local.

Set $F=\Rhom RBC$ and $G=\Rhom RCB$.  In view of Lemma \ref{lem:bass},
the isomorphism $B\simeq \Rhom RFC$ and $C\simeq \Rhom RGB$ yield
  \[
I_{R}^{B}(t)=P^{R}_{F}(t)\cdot I_{R}^{C}(t)\quad
  \text{and} \quad
I_{R}^{C}(t)=P^{R}_{G}(t) \cdot I_{R}^{B}(t)
  \]
As $I_{R}^{B}(t)\ne 0$ holds, see \ref{supp:bass}, these equalities imply
  \(
P^{R}_{F}(t)\cdot P^{R}_{G}(t)=1\,,
  \)
hence $P^{R}_{F}(t)=t^{r}$ holds for some $r$. Proposition \ref{prop:semid}
now gives $F\simeq\Shift^{r}R$, so one gets
  \[
B\simeq \Rhom RFC\simeq \Rhom R{\Shift^{r}R}C\simeq\Shift^{-r}C\,.
  \]
Thus, $B$ is derived $B$-reflexive, hence semidualizing by
Proposition~\ref{prop:semidualizing}.   A direct verification
shows that the following evaluation map is an isomorphism:
  \[
\Rhom R{\Shift^{-r}C}C \dtensor R{\Shift^{-r}C}\to C\,.
  \]

(ii)$\implies$(iii) Lemma~\ref{refl:operations2} shows that $C$ is
semidualizing; the rest is clear.

(iii)$\implies$(i).  Proposition~\ref{prop:semidualizing} shows that
$B$ satisfies $\Supp_RB=\Spec R$ and is derived $B$-reflexive. {From}
Lemma~\ref{refl:operations} we then see that $B$ is derived
$C$-reflexive. A second loop, this time starting from $C$, shows that $C$
is derived $B$-reflexive.
  \end{proof}

Taking $B=R$ one recovers a result contained in \cite[8.3]{Ch2}.

\begin{corollary}
\label{cor:reflsemid}
A complex in $\dcat[R]$ is invertible if and only if it is semidualizing
and derived $R$-reflexive.\qed 
\end{corollary}

  \section{Duality}
   \label{Duality}

We say that a contravariant $R$-linear exact functor $\ssd\col\dcat[R]\to\dcat[R]$
is a \emph{duality} on a subcat\-egory~$\sA$ of~$\dcat[R]$ if it satisfies
$\ssd(\sA)\subseteq\sA$ and $\ssd^2\vert_{\sA}$ is isomorphic to $\id^\sA$.

In this section we link dualities on subcategories of $\dcatb R$ to semidualizing complexes.  In the `extremal' cases, when the subcategory equals $\dcatb R$ 
itself or when the semidualizing complex is the module $R$, we recover a 
number of known results and answer some open questions.

\subsection{Reflexive subcategories}
  \label{Reflexive subcategories}

For each complex $C$ in $\dcat[R]$, set 
 \[
\sh_C=\Rhom R-C\col\dcat[R]\lra\dcat[R]\,.
 \]
The \emph{reflexive subcategory} of $C$ is the full subcategory of 
$\dcat[R]$ defined by
  \[
\refl=\{M\in\dcatb R\mid M\simeq\sh_C^2(M)\}\,.
 \]
By Proposition \ref{prop:foxby}, the functor $\sh_C$ is a duality on $\refl$ 
\emph{provided} $\sh_C(\refl)\subseteq\refl$ holds. We note that, 
under an additional condition, such a $C$ has to be semidualizing.

  \begin{ssubproposition}
    \label{prop:dualitiesd}
Let $\ssd$ be a duality on a subcategory $\sA$ of $\dcatb R$. 

If $\sA$ contains $R$, then the complex $C=\ssd(R)$ is semidualizing and $\sA$ is contained in $\refl[C]$; furthermore, for each module $R$-module $M$ in $\sA$ there is an isomorphism 
\[
M\simeq \Rhom R{\ssd(M)}C\,.
\]
  \end{ssubproposition}

  \begin{proof}
Let $M$ be an $R$-module.  For each $n\in\BZ$ one then has isomorphisms
  \begin{align*}
\Ext nR{\ssd(M)}C
&\cong \Hom{\dcat[R]}{\ssd(M)}{\Shift^{n}C}\\
&\cong\Hom{\dcat[R]}R{\Shift^{n}\ssd^2(M)}\\
&\cong \Hom {\dcat[R]}R{\Shift^{n}M}\\
&\cong \Ext nRRM\\
&\cong 
  \begin{cases}
M &\text{for }n=0\,;\\
0 &\text{for }n\ne0\,.
  \end{cases}
   \end{align*}
It follows that $\Rhom R{\ssd(M)}C$ is isomorphic to $M$ in $\dcat[R]$. For $M=R$ 
this yields $\Rhom RCC \simeq R$, so $C$ is semidualizing by Proposition~\ref{prop:semidualizing}. 
  \end{proof}

Next we show that semidualizing complexes do give rise to dualities and that,
furthermore, they are determined by their reflexive subcategories:

  \begin{ssubtheorem}
    \label{thm:dualitiesC}
Let $C$ be a semidualizing complex for $R$.

The functor $\sh_C$ is a duality on $\refl$, the natural transformation 
$\delta^C\col\id\to\sh_C^2$ restricts to an isomorphism of functors on $\refl$,
and $R$ is in $\refl$.

A complex $B$ in $\dcatb R$ satisfies $\refl[B]=\refl$ if and only if $B$ is 
derived Picard equivalent to $C$ (in which case $B$ is semidualizing). 
  \end{ssubtheorem}

  \begin{proof}
Theorem \ref{thm:bounded-reflexive} implies that $\sh_C$ takes values in 
$\refl$ and that $\delta^C$ restricts to an isomorphism on $\refl$, while 
Proposition \ref{prop:semidualizing} shows that $R$ and $C$ are in $\refl$.

The last assertion results from Theorem \ref{thm:reflsemid}.
  \end{proof}

The preceding results raise the question whether every duality functor on 
a subcategory of $\dcatb R$ is representable on its reflexive subcategory.  

\subsection{Dualizing complexes}
  \label{Dualizing complexes}
Let $D$ be a complex in $\dcat[R]$.  

Recall that $D$ is said to be \emph{dualizing} for $R$ if it is semidualizing and of finite injective dimension.  If $D$ is dualizing, then $\refl[D]=\dcatb R$; see \cite[p.\,258, 2.1]{H}.

In the language of Hartshorne~\cite[p.\,286]{H}, the complex $D$ is \emph{pointwise dualizing} for $R$ if it is in $\dcatn R$ and the complex $D_\fp$ is dualizing for $R_\fp$ for each $\fp\in\Spec R$. When in addition $D$ is in $\dcatb R$ we say that it is \emph{strongly pointwise dualizing}; this terminology is due to Gabber; see~\cite[p.\,120 ]{Co}, also for discussion on why the latter concept is the more appropriate one. 

For a different treatment of dualizing complexes, see Neeman \cite{Ne}.

The next result is classical, see \cite[p.\,283, 7.2; p.\,286, Remark 1; p.\,288, 8.2]{H}:

  \begin{ssubchunk}
    \label{dualizing_classical}
Let $D$ be a complex in $\dcatb R$. The complex $D$ is dualizing if and 
only if it is pointwise dualizing and $\dim R$ is finite.  
 \end{ssubchunk}

The equivalence of conditions (i) and (ii) in the next result is due to Gabber, see 
\cite[3.1.5]{Co}.  Traces of his argument can be found in our 
proof, as it refers to Theorem~\ref{thm:bounded-reflexive}, and
thus depends on Theorem~\ref{thm:rfd}. 

  \begin{ssubtheorem}
    \label{thm:gabber}
For $D$ in $\dcat[R]$ the following conditions are equivalent.
 \begin{enumerate}[\quad\rm(i)]
  \item
$D$ is strongly pointwise dualizing for $R$.
  \item
$\sh_{D}$ is a duality on $\dcatb{R}$.
  \item
$D$ is in $\dcatb R$, and for each $\fm\in\Max R$ and finite $R$-module $M$ one has
  \[
M_{\fm}\simeq \Rhom{R_{\fm}}{\Rhom{R_{\fm}}{M_{\fm}}{D_{\fm}}}{D_{\fm}}
\quad\text{in}\quad \dcat[R_{\fm}]\,.
  \]
   \end{enumerate}
\end{ssubtheorem}

  \begin{proof}
(i)$\implies$(iii). By definition, $D\in\dcatb R$  and
$D_\fm$ is dualizing for $R_\fm$. Moreover, it is clear that
$M_\fm\in\dcatb{R_\fm}=\refl[D_\fm]$.

(iii)$\implies$(i). Let $\fm$ be a maximal ideal of $R$.  For
$M_\fm=R_\fm$ the hypothesis implies that $D_\fm$ is semi\-dualizing,
see Proposition~\ref{prop:semidualizing}. For $M=R/\fm$ it implies, by
the first part of Lemma~\ref{lem:bass}, that $\Rhom{R_\fm}{R_\fm/\fm
R_\fm}{D_\fm}\in\dcatb{R_\fm}$; this means that $D_\fm$ has finite
injective dimension over $R_\fm$, see \ref{Bass series}. Localization shows
that $D_\fp$ has the corresponding properties for every prime ideal $\fp$
of $R$, contained in $\fm$.

(iii)$\iff$(ii).  
The complex $D$ is semidualizing---by Proposition~\ref{prop:semidualizing}
if (iii) holds, by Proposition~\ref{prop:dualitiesd} if (ii) holds;
so the equivalence results from Theorem~\ref{thm:bounded-reflexive}.
   \end{proof}


  \begin{ssubcorollary}
    \label{cor:goto_derived}
The ring $R$ is Gorenstein if and only if the complex $R$ is strongly pointwise
dualizing, if and only if each complex in $\dcatb R$ is derived 
$R$-reflexive.
 \end{ssubcorollary}

 \begin{proof}
For arbitrary $R$ and $\fp\in\Spec R$, the complex $R_\fp$ is
semi-dualizing for $R_\fp$.  Thus, the first two conditions are
equivalent because---by definition---the ring $R$ is Gorenstein if
and only if $R_\fp$ has a finite injective resolution as a module over
itself for each $\fp$.  The second and third conditions are equivalent
by Theorem~\ref{thm:gabber}.
 \end{proof}

Given a homomorphism $R\to S$ of rings, recall that $\Rhom RS-$ is
a functor from $\dcat[R]$ to $\dcat$. The next result is classical,
cf.~\cite[p.\,260, 2.4]{H}.

  \begin{ssubcorollary}
    \label{cor:kawasaki}
If $R\to S$ is a finite homomorphism of rings and $D\in\dcatb
R$ is pointwise dualizing for $R$, then $\Rhom RSD$ is  pointwise
dualizing for $S$.
 \end{ssubcorollary}

 \begin{proof}
Set $D'= \Rhom RSD$.  For each $M$ in $\dcatb S$ one has
  \[
\Rhom RMD\simeq \Rhom SM{D'}
 \quad\text{in}\quad \dcat\,.
  \]
It shows that $\Rhom SM{D'}$ is in $\dcatb S$, and that
the restriction of $\sh_D$ to $\dcatb S$ is equivalent to
$\sh_{D'}$. Theorem~\ref{thm:gabber} then shows that $D'$ is pointwise
dualizing.
 \end{proof}

It follows from Corollaries~\ref{cor:goto_derived} and \ref{cor:kawasaki} that if 
$S$ is a homomorphic image of a Gorenstein ring, then it admits a strongly pointwise 
dualizing complex. Kawasaki \cite[1.4]{Kw} proved that if $S$ has a 
\emph{dualizing complex}, then $S$ is a homomorphic image of some 
Gorenstein ring \emph{of finite Krull dimension}, so we ask:

 \begin{ssubquestion}
Does the existence of a strongly pointwise dualizing complex for $S$ imply that $S$ is a homomorphic image of some Gorenstein ring?
 \end{ssubquestion}

  \subsection{Finite G-dimension}
    \label{Gorenstein dimension}

The category $\refl[R]$ of derived $R$-reflexive complexes contains all perfect complexes, but may be larger.  To describe it we use a notion from module theory:  An $R$-module $G$ is \emph{totally reflexive} when 
it is finite,
\begin{alignat*}{2}
   \Hom R{\Hom RGR}R&\cong G &\quad &\text{and}\\
   \Ext nR{\Hom RGR}R&=0=\Ext nRGR &\quad &\text{for all}\quad n\ge1\,.
\end{alignat*}

A complex of $R$-modules is said to have \emph{finite G-dimension} (for 
\emph{Gorenstein dimension}) if it is quasi-isomorphic to a bounded complex 
of totally reflexive modules. The study of modules of finite $G$-dimension was 
initiated by Auslander and Bridger~\cite{AB}. The next result, taken from 
\cite[2.3.8]{Ch1}, is due to Foxby:

 \begin{ssubchunk}
   \label{gd=reflexive}
A complex in $\dcat[R]$ is in $\refl[R]$ if and only if it has finite $G$-dimension.
 \end{ssubchunk}

Theorems  \ref{thm:refl} and \ref{thm:bounded-reflexive} specialize to:

\begin{ssubtheorem}
\label{thm:Gperfect}
  \pushQED{\qed}
For a complex $M\in\dcatb R$ the following are equivalent.
\begin{enumerate}[\quad\rm(i)]
  \item
$M$ is derived $R$-reflexive.
  \item
$\Rhom RMR$ is derived $R$-reflexive.
  \item
For each $\fm\in\Max R$ there is an isomorphism
  \[
M_{\fm}\simeq \Rhom{R_{\fm}}{\Rhom{R_{\fm}}{M_{\fm}}{R_{\fm}}}{R_{\fm}}
\quad\text{in}\quad \dcat[R_{\fm}]\,.
  \]
  \item
$U^{-1}M$ is derived $U^{-1}R$-reflexive for each multiplicatively closed set $U$.\qed
  \end{enumerate}
    \end{ssubtheorem}

Combining \ref{gd=reflexive} and Corollary~\ref{cor:goto_derived}, we obtain a 
new proof of a result due to Auslander and Bridger \cite[4.20]{AB} (when $\dim R$ 
is finite) and to Goto \cite{Go} (in general):

\begin{ssubcorollary}
 \label{cor:goto_original}
The ring $R$ is Gorenstein if and only if every finite $R$-module has finite G-dimension. \qed
 \end{ssubcorollary}

It is easy to check that if a complex $M$ has finite $G$-dimension over $R$, then so does the complex of $R_{\fp}$-modules $M_{\fp}$, for any prime ideal $\fp$. Whether the converse holds  had been an  open question, which we settle as a corollary of \ref{gd=reflexive} and Theorem~\ref{thm:Gperfect}:

\begin{ssubcorollary}
\label{cor:localGdim}
A homologically finite complex $M$ has finite G-dimension if\/ $($and only if\/$)$ the
complex $M_\fm$ has finite G-dimension over $R_\fm$ for every $\fm\in \Max R$. \qed
 \end{ssubcorollary}

\section{Rigidity}
 \label{Rigidity}

Over any commutative ring, we introduce a concept of rigidity of one 
complex relative to another, and establish the properties responsible for the name. 
In \S\ref{Relative rigidity} we show how to recover the notion of rigidity for complexes 
over commutative algebras, defined by Van den Bergh, Yekutieli and Zhang. 

Let $C$ be a complex in $\dcat[R]$.  We say that a complex $M$ in $\dcat[R]$ is
\emph{$C$-rigid} if there exists an isomorphism
 \begin{equation}
   \label{eq:rigid}
\mu\col M\xra{\,\simeq\,}\Rhom R{\Rhom RMC}M\quad\text{in}\quad \dcat 
[R]\,.
  \end{equation}
In such a case, we call $\mu$ a \emph{$C$-rigidifying isomorphism} 
and $(M,\mu)$ a $C$-\emph{rigid pair}.

\begin{example}
\label{exa:rigidity}
Let $C$ be a semidualizing complex.  For each idempotent element 
$a\in R$, using \eqref{eq:homothety} and \ref{ch:idempotents} 
one obtains a canonical composite isomorphism
 \begin{gather*}
\xymatrixcolsep{6.2pc}
\xymatrixrowsep{.2pc}
\xymatrix {
\gamma_a\col C_a\:\ar@{->}[r]^-{\simeq}
&\,\Rhom RR{C_a}{\phantom{xxxxxxxxxx}}
\\
{\phantom{\gamma_a\col C_a}\:}
\ar@{->}[r]^-{\ \Rhom R{\chi^{C}}{C_a}^{-1}\,}
&\:\Rhom R{\Rhom RCC}{C_a}{\phantom{x}}
\mspace{-4mu} \\
{\phantom{\gamma_a\col C_a}\:}\ar@{->}[r]^-{\simeq}
&\:\Rhom R{\Rhom R{C_a\oplus C_{1-a}}{\>C}}{C_a}\,
\mspace{-68.5mu}\\
{\phantom{\gamma_a\col C_a}\:}\ar@{->}[r]^-{\simeq}
&\;\Rhom R{\Rhom R{C_a\>}C}{C_a}\,.\mspace{-11mu}
}
 \end{gather*}
Thus, for each idempotent $a$ there exists a \emph{canonical $C$-rigid pair}
$(C_a\>,\gamma_a)$.
  \end{example}

\begin{theorem}
\label{thm:rigidity}
Let $C$ be a semidualizing complex.

A complex $M\in\dcatb R$ is $C$-rigid if and only if it satisfies
  \begin{equation}
  \label{eq:rigidifier}
M\simeq C_a\quad\text{in}\quad \dcat[R]
 \end{equation}
for some idempotent $a$ in $R$; such an idempotent is
determined by the condition
  \begin{equation}
  \label{eq:locus}
\Supp_RM=\{\fp\in\Spec R\mid \fp\not\ni a\}\,.
 \end{equation}
  \end{theorem}

\begin{proof}
The `if' part comes from Example \ref{exa:rigidity}, so assume that $M 
$ is $C$-rigid.

Set $L=\Rhom RMC$ and let $M\simeq\Rhom RLM$ be a rigidifying isomorphism. Theorem~\ref{thm:invertible} produces a unique idempotent $a$ in $R$ satisfying \eqref{eq:locus}, and such that the complex $L_a$ is invertible in $\dcat[R_a]$. Hence, $L_{a}$ is derived $C_{a}$-reflexive in $\dcat[R_{a}]$ by Lemma~\ref{refl:operations}.   Thus, $\Rhom{R_{a}}{M_{a}}{C_{a}}$ is derived $C_{a}$-reflexive, and hence so is
$M_a$, by Theorem~\ref{thm:bounded-reflexive}. This explains the second isomorphism below:
  \[
\Rhom {R_{a}}{L_{a}}{C_{a}}
 \simeq\Rhom {R_a}{\Rhom {R_a}{M_a}{C_a}}{C_a}
\simeq M_{a}\simeq\Rhom{R_{a}}{L_{a}}{M_{a}}\,.
  \]
The third one is a localization of the rigidifying isomorphism.
Consequently $M_{a}\simeq C_{a}$ in $\dcat[R_{a}]$; see \ref{picard}. It
remains to note that one has $M\simeq M_a$ in $\dcat[R]$; see
\ref{ch:idempotents}.
 \end{proof}

A \emph{morphism of $C$-rigid pairs} is a commutative diagram
\[
\xymatrixcolsep{4.5pc}
\xymatrixrowsep{2.5pc}
(\alpha)=
\begin{gathered}
\xymatrix {
M\ar@{->}[r]^-{\mu}\ar@{->}[d]^{\alpha}
&\Rhom R{\Rhom RMC}M\ar@{->}[d]^{\Rhom R{\Rhom R{\alpha}C}{\alpha}}
 \\
N\ar@{->}[r]^-{\nu}
&\Rhom R{\Rhom RNC}N
}
\end{gathered}
\]
in $\dcat[R]$.  The $C$-rigid pairs and their morphisms form a category, where composition is given by $(\beta)(\alpha)=(\beta\alpha)$ and $\id^{(M,\>\mu)}=(\id^{M})$.

The next result explains the name `rigid complex'.  It is deduced from 
Theorem~\ref{thm:rigidity} by transposing a beautiful observation 
of Yekutieli and Zhang from the proof of \cite[4.4]{YZ1}:  A morphism 
of rigid pairs is a natural isomorphism from a functor in $M$ that is linear 
to one that is quadratic, so it must be given by an idempotent.

\begin{theorem}
\label{thm:uniquerigidity}
If\/ $C$ is a semidualizing complex and $(M,\mu)$ and $(N,\nu)$ are
$C$-rigid pairs in $\dcatb R$,  then the following conditions are  
equivalent.
 \begin{enumerate}[\quad\rm(i)]
\item
There is an equality $\Supp_R N=\Supp_R M$.
\item
There is an isomorphism $M\simeq N$ in $\dcat[R]$.
\item
There is  a unique isomorphism of\/ $C$-rigid pairs $(M,\mu)\simeq(N, 
\nu)$.
\end{enumerate}
 \end{theorem}

 \begin{proof}
(i)$\implies$(iii).
Let $\alpha\col C_a\xra {\simeq} M$ be an isomorphism in $\dcat[R]$
given by~\eqref{eq:rigidifier}, with $a$ the idempotent defined by 
formula \eqref{eq:locus}.   It suffices to prove that $(M,\mu)$ is uniquely 
isomorphic to the $C$-rigid pair $(C_a,\gamma_a)$ from 
Example~\eqref{exa:rigidity}.  Since it is equivalent to prove the same
in $\dcat[R_a]$, we may replace $R$ by~$R_a$ and drop all references
to localization at $\{1,a\}$.

Set $\wt\alpha=\Rhom R{\Rhom R{\alpha}C}{\alpha}$: this is an isomorphism,
and hence so is $\alpha^{-1}\circ\mu^{-1}\circ\wt\alpha\circ\gamma\col C\to C$.
As $C$ is semidualizing, there is an isomorphism
 \[
\HH0{\chi^C}\col R\xra{\,\cong\,}\HH0{\Rhom RCC}=\Hom{\dcat[R]}CC\,,
 \]
of rings, so $\alpha^{-1}\circ\mu^{-1}\circ\wt\alpha\circ\gamma=\HH0{\chi^C}(u)$
for some unit $u$ in $R$. The next computation shows that
$(u^{-1}\alpha)\col (C,\gamma)\to (M,\mu)$ is an isomorphism
of $C$-rigid pairs:
 \begin{align*}
\Rhom R{\Rhom R{u^{-1}\alpha}C}{u^{-1}\alpha}\circ\gamma
&=u^{-2}(\wt\alpha\circ\gamma)\\
&=u^{-2}\cdot u(\mu\circ\alpha)\\
&=\mu\circ(u^{-1}\alpha)\,.
 \end{align*}

Let $(\beta)\col (C,\gamma)\to (M,\mu)$ also be such an isomorphism.
The isomorphism $\HH0{\chi^C}$ implies that in $\dcat[R]$ one has
$\beta^{-1}\circ u^{-1}\alpha=v\id^{C}$ for some unit $v\in R$, whence
$v\id^{C}$ is a rigid endomorphism of the rigid pair $(C,\gamma)$. Thus
\begin{align*}
v\gamma
&=\gamma\circ(v\id^C)\\
&=\Rhom R{\Rhom R{v\id^C}C}{v\id^C}\circ\gamma\\
&=v^2\Rhom R{\Rhom R{\id^C}C}{\id^C}\circ\gamma\\
&=v^2\gamma\,.
 \end{align*}
As $v$ and $\gamma$ are invertible one gets $(v-1)\id^C=0$,
hence $v-1\in\ann_RC=0$.  This gives $v=1$, from where one
obtains $\beta^{-1}\circ u^{-1}\alpha=\id^{C}$, and finally $
(\beta)=(u^{-1}\alpha)$.

(iii)$\implies$(ii)$\implies$(i).  These implications are evident.
 \end{proof}

An alternative formulation of the preceding result is sometimes useful.

\begin{remark}
\label{rem:altrigidity}
Let $(M,\mu)$ be a $C$-rigid pair in $\dcatb R$, and $N$ a complex in 
$\dcatb R$.

For each isomorphism $\alpha\col N\xra{\simeq}M$ in $\dcat[R]$, set
\[
\rho(\alpha)=  (\Rhom R{\Rhom R{\alpha}C}{\alpha})^{-1}\circ \mu \circ \alpha\,;
\]
this is a morphism from $N$ to $\Rhom R{\Rhom R NC}N$.

Theorem~\ref{thm:uniquerigidity} shows that the assignment 
$\alpha\mapsto (N,\rho(\alpha))$ yields a bijection
\[
\{\text{isomorphisms from $N$ to $M$}\}\leftrightarrow
\{\text{rigid pairs $(N,\nu)$ isomorphic to $(M,\mu)$}\}
\]
\end{remark}

We finish with a converse, of sorts, to Example~\ref{exa:rigidity}.

\begin{proposition}
\label{rigid:sdc}
If $C$ in $\dcatb R$ is $C$-rigid, then there exist an idempotent
$a$ in~$R$, a semidualizing complex $B$ for $R_{a}$, and an
isomorphism $C\simeq B$ in $\dcat[R]$.
  \end{proposition}

\begin{proof}
One has $C\simeq \Rhom R{\Rhom RCC}C$ by hypothesis.
Theorem~\ref{thm:invertible} and \ref{ch:idempotents} provide an 
idempotent $a\in R$, such that the $R_{a}$-module 
$\HH0{\Rhom RCC_a}$ is invertible and in 
$\dcat[R]$ there are natural isomorphisms $C\simeq C_a$ and
  \[
\HH0{\Rhom RCC_a}
\simeq\Rhom RCC_a
\simeq\Rhom {R_a}{C_a}{C_a}\,.
  \]
It follows that the homothety map
\[
\chi\col R_a\to \Hom{\dcat[R_a]}{C_a}{C_a}\cong
\HH0{\Rhom{R_a}{C_a}{C_a}}
\] 
turns $\Hom{\>\textsf D(R_a)}{C_a}{C_a}$ into \emph{both} an invertible 
$R_{a}$-module \emph{and} an $R_a$-algebra. Localizing at prime ideals 
of $R_a$, one sees  that such a $\chi$ must be an isomorphism; so the 
proposition holds with $B=C_a$.
 \end{proof}
 
\section{Relative dualizing complexes}
\label{Relative dualizing complexes}

In this section $K$ denotes a commutative noetherian ring, $S$ a commutative 
ring, and $\sigma\col K\to S$ a homomorphism of rings that is assumed to be 
\emph{essentially of finite type}:  This means that $\sigma$ can be factored as 
a composition
 \begin{equation}
   \label{eq:reduce1}
K\hra K[x_1,\dots,x_e]\to W^{-1}K[x_1,\dots,x_e]=Q\tra S
 \end{equation}
of homomorphisms of rings, where $x_1,\dots,x_e$ are indeterminates, $W$ is a 
multiplicatively closed set, the first two maps are canonical, the equality defines
$Q$, and the last arrow is surjective; the map $\sigma$ is \emph{of finite type} if 
one can choose $W=\{1\}$.  

As usual, $\Omega_{Q|K}$ stands for the $Q$-module of K\"ahler differentials;
for each $n\in\BZ$ we set $\Omega^n_{Q|K}={\ts\bwedge}^{n}_Q\Omega_{Q|K}$.  
Fixing the factorization \eqref{eq:reduce1}, we define a \emph{relative dualizing 
complex} for $\sigma$ by means of the following equality:
 \begin{equation}
   \label{eq:realative1}
D^{\sigma}=\Shift^e\Rhom QS{\Omega^e_{Q|K}}\,.
 \end{equation}

Our goal here is to determine when $D^\sigma$ is semidualizing, invertible, or 
dualizing.  It turns out that each one of these properties is equivalent to some property 
of the homomorphism $\sigma$, which has been studied earlier in a different context.  
We start by introducing notation and terminology that will be used throughout the section.

For every $\fq$ in $\Spec S$ we let $\fq\cap K$ denote the prime ideal $\sigma^{-1}(\fq)$
of $K$, and write $\sigma_{\fq}\col K_{\fq\cap K}\to S_{\fq}$ for the induced local 
homomorphism; it is essentially of finite type.

Recall that a ring homomorphism $\dot\sigma\col K\to P$ is said to be (\emph{essentially}) \emph{smooth} if it is (essentially) of finite type, flat, and for each ring homomorphism $K\to k$, where $k$ is a field, the ring $k \otimes_KP$ is regular; by \cite[17.5.1]{Gr} this notion of smoothness is equivalent to the one defined in terms of lifting of homomorphisms. When $\dot\sigma$ is essentially smooth $\Omega_{P|K}$ is finite projective over $P$; in case $\Omega_{P|K}$ has rank $d$, see \ref{ch:projrank}, we say that $\dot\sigma$ has \emph{relative dimension} $d$. The $P$-module $\Omega^d_{P|K}$ is then invertible.

An (\emph{essential}) \emph{smoothing} of $\sigma$ (of relative dimension $d$) is a 
decomposition
  \begin{equation}
 \label{eq:independence}
K\xra{\dot\sigma} P\xra{\sigma'} S
 \end{equation}
of $\sigma$ with $\dot\sigma$ (essentially) smooth of fixed relative dimension (equal
to $d$) and $\sigma'$ \emph{finite}, meaning that $S$ is a finite $P$-module via 
$\sigma'$; an essential smoothing of $\sigma$ always exists, see \eqref{eq:reduce1}.

\subsection{Basic properties}
  \label{ssec:rdc}
Fix an essential smoothing \eqref{eq:independence} of relative 
dimension $d$.

\begin{ssubchunk}
\label{ch:independence}
By \cite[1.1]{AILN}, there exists an isomorphism
\[
D^{\sigma}\simeq\Shift^d\Rhom PS{\Omega^d_{P|K}} 
\quad\text{in}\quad \dcat\,.
\]
 \end{ssubchunk}

\begin{ssubchunk}
\label{eq:basechange}
For each $M$ in $\dcatb S$ there are isomorphisms
\[
\begin{aligned}
\Rhom SM{D^\sigma}
&=\Rhom SM{\Shift^d\Rhom PS{\Omega^d_{P|K}}}\\
&\simeq\Shift^{d}\Rhom PM{\Omega^d_{P|K}}\\
&\simeq \Rhom PMP\otimes_P{\Shift^{d}\Omega^d_{P|K}}
   \end{aligned}
\]
in $\dcat$, because $\Omega^d_{P|K}$ is an invertible $P$-module.
 \end{ssubchunk}
 
\begin{ssubproposition}
\label{prop:naturality1}
If $U\subseteq K$ and $V\subseteq S$ are multiplicatively closed sets satisfying $\sigma(U)\subseteq V$, and $\wt\sigma\col U^{-1}K\to V^{-1}S$ is the induced map, then one has
 \[
D^{\wt\sigma}\simeq V^{-1}D^{\sigma}\quad\text{in}\quad\dcat[V^{-1}S]\,.
 \]
\end{ssubproposition}

\begin{proof}
Set $V'{}=\sigma'^{-1}(V)$.  In the induced factorization 
$U^{-1}K\to (V'){}^{-1}P\to V^{-1}S$ of $\wt\sigma$ the first map is essentially 
smooth of relative dimension $d$ and the second one is  finite.  The first 
and the last isomorphisms in the next chain hold by \ref{ch:independence}, the rest because 
localization commutes with modules of differentials and exterior powers:
\begin{align*}
D^{\wt\sigma}
&\simeq \Shift^{d}\Rhom {(V'){}^{-1}P}{(V'){}^{-1}S}{\Omega_{(V'){}^ 
{-1}P|U^{-1}K}^{d}}\\
&\simeq \Shift^{d}\Rhom {(V'){}^{-1}P}{(V'){}^{-1}S}{(V'){}^{-1} 
\Omega_{P|K}^{d}}\\
&\simeq (V'){}^{-1}\Shift^{d}\Rhom PS{\Omega_{P|K}^{d}}\\
&\simeq V^{-1}D^{\sigma}\,.
 \qedhere
\end{align*}
  \end{proof}

\begin{ssubproposition}
\label{prop:naturality2}
If $\vf\col S\to T$ is a finite homomorphism of rings, then for 
the map $\tau=\vf\sigma\col K\to T$ there is an isomorphism
\[
D^{\tau}\simeq\Rhom ST{D^\sigma}\quad\text{in}\quad\dcat[T]\,.
\]
\end{ssubproposition}

\begin{proof}
The result comes from the following chain of isomorphisms:
\begin{align*}
D^{\tau}
 & \simeq  \Shift^d\Rhom PT{\Omega_{P|K}^{d}}\\
  &\simeq  \Rhom ST{\Shift^d\Rhom PS{\Omega_{P|K}^{d}}}\\
  &=  \Rhom ST{D^\sigma}\,,
\end{align*}
where the first one is obtained from the factorization 
$K\xra{\kappa}P\xra{\vf\sigma'}T$ of $\tau$ and the second
one by adjunction.
 \end{proof}

\subsection{Derived $D^\sigma$-reflexivity}
\label{derived reflexivity}

A standard calculation shows that derived $D^\sigma$-reflexivity can be read 
off any essential smoothing, see \eqref{eq:independence}:

\begin{ssubproposition}
   \label{prop:relativeG-dim}
A complex $M$ in $\dcat$ is derived $D^\sigma$-reflexive if and only
if $M$ is derived $P$-reflexive when viewed as a complex in $\dcat[P]$.
  \end{ssubproposition}

 \begin{proof}
Evidently, $M$ is in $\dcatb S$ if and only if it is in $\dcatb P$.
{From} \ref{eq:basechange} one sees that $\Rhom SM{D^\sigma}$ is in $\dcatb S$ if 
and only if $\Rhom PMP$ is in $\dcatb P$. 

Set $\Omega=\Shift^d\Omega^d_{P|K}$, where $d$ is the relative dimension 
of  $K\to P$, and let $\Omega\to I$ be a semiinjective resolution in $\dcat[P]$.
Thus, $D^{\sigma}$ is isomorphic to $\Hom PSI$ in $\dcat$.  The biduality 
morphism $\delta^{\Omega}_{M}$ in $\dcat[P]$ is realized by a morphism
  \[
M\to \Hom P{\Hom PMI}I
  \]
of  complexes of $S$-modules;  see \eqref{eq:biduality}. 
Its composition with the natural isomorphism of complexes of $S$-modules
\[
\Hom P{\Hom PMI}I \cong \Hom S{\Hom SM{\Hom PSI}}{\Hom PSI}
\]
represents the morphism $\delta^{D^{\sigma}}_{M}$ in $\dcat$.  It follows
that $M$ is derived $D^\sigma$-reflexive if and only if it is derived 
$\Omega$-reflexive. Since $\Omega$ is an invertible $P$-module, the last 
condition is equivalent---by Lemma~\ref{refl:operations}---to the derived 
$P$-reflexivity of $M$.
 \end{proof}

A complex $M$ in $\dcatP S$ is said to have \emph{finite flat dimension}
over $K$ if $M$ is isomorphic in $\dcat[K]$ to a bounded complex of flat
$K$-modules; we then write $\fd_KM<\infty$.

When $\fd_KS$ is finite we say that $\sigma$ is of \emph {finite flat
dimension} and write $\fd\sigma<\infty$.

 \begin{ssubchunk}
   \label{fpd:smooth}
A complex $M$ in $\dcatb S$ satisfies $\fd_{K}M<\infty$ if and only if it is perfect in $\dcat[P]$ for some (equivalently, any) factorization \eqref{eq:independence} of $\sigma$; see \cite[beginning of \S6]{AILN}.
 \end{ssubchunk}

\begin{ssubcorollary}
   \label{cor:relativeP-dim}
A complex $M$ in $\dcatb S$ with $\fd_KM<\infty$ is derived
$D^\sigma$-reflexive.
  \end{ssubcorollary}

 \begin{proof}
By \ref{fpd:smooth} the complex $M$ is perfect in $\dcat[P]$.  It is then
obviously derived $P$-reflexive, and so is derived $D^\sigma$-reflexive
by the previous  proposition.
 \end{proof}

\subsection{Gorenstein base rings}
\label{Gorenstein base rings}
Relative dualizing complexes and their absolute counterparts, see
\ref{Dualizing complexes}, are compared in the next result, where the
`if' part is classical.

  \begin{ssubtheorem}
 \label{thm:gorring}
The complex $D^{\sigma}$ is strongly pointwise dualizing for $S$ if and only if
the ring $K_{\fq\cap K}$ is Gorenstein for every prime ideal $\fq$ of $S$.
  \end{ssubtheorem}

  \begin{proof}
Factor $\sigma$ as in \eqref{eq:reduce1} and set $\fp=\fq\cap K$.  The homomorphism $\sigma_{\fq}\col K_{\fp}\to S_\fq$ satisfies $(D^{\sigma})_{\fq}\cong D^{\sigma_{\!\fq}}$ by Proposition~\ref{prop:naturality1}.  Localizing, we may assume that $\sigma$ is a local homomorphism $(K,\fp)\to (S,\fq)$, and that the ring
$Q$ is local. As the ring $Q/\fp Q$ is regular, $K$ is Gorenstein if and only so is $Q$; see \cite[23.4]{Ma}. Thus, replacing $Q$ with $K$ we may further assume that $\sigma$ is surjective.

If $K$ is Gorenstein, then $D^{\sigma}=\Rhom KSK$ holds so it is dualizing for $S$ by Corollaries~\ref{cor:goto_derived} and \ref{cor:kawasaki}.

When $D^{\sigma}$ is dualizing for $S$, the residue field $k=S/\fq$ is derived $D^{\sigma}$-reflexive, see Theorem \ref{thm:gabber}. By Proposition~\ref{prop:relativeG-dim} it is also derived $K$-reflexive, which implies $\Ext nKkK=0$ for $n\gg0$. Thus, $K$ is Gorenstein; see \cite[18.1]{Ma}.
 \end{proof}

\subsection{Homomorphisms of finite G-dimension}
 \label{Finite G-dimension}

When the $P$-module $S$ has finite G-dimension, see
\ref{Gorenstein dimension}, we say that $\sigma$ has \emph{finite
$G$-dimension} and write $\gdim\sigma<\infty$. By the following result,
this notion is independent of the choice of factorization.

\begin{ssubproposition}
   \label{prop:relative-semidualizing}
The following conditions are  equivalent.
 \begin{enumerate}[\rm\quad(i)]
\item
$D^{\sigma}$ is semi-dualizing for $S$.
\item 
$\sigma$ has finite G-dimension.
\item
$\sigma_\fn$ has finite G-dimension for each $\fn\in\Max S$.
\end{enumerate}
  \end{ssubproposition}

 \begin{proof}
(i)$\iff$(ii).  By Proposition~\ref{prop:semidualizing},
$D^{\sigma}$ is semi-dualizing for $S$ if and only if $S$ is derived
$D^{\sigma}$-reflexive.  By Proposition~\ref{prop:relativeG-dim} this is
equivalent to $S$ being derived $P$-reflexive in $\dcat[P]$, and hence,
by \ref{gd=reflexive}, to $S$ having finite G-dimension over $P$.

(ii)$\iff$(iii).  Proposition~\ref{prop:naturality1} yields an
isomorphism $D^{\sigma_{\fn}}\simeq (D^{\sigma})_{\fn}$ for each
$\fn$. Given (i)$\iff$(ii), the desired equivalence follows from
Proposition~\ref{prop:semidualizing}.  \end{proof}

Combining the proposition with Theorem~\ref{thm:gorring} and Corollary 
\ref{cor:relativeP-dim}, one obtains:

 \begin{ssubcorollary}
   \label{cor:relative-semidualizing}
Each condition below implies that $\sigma$ has finite G-dimension:
 \begin{enumerate}[\quad\rm(a)]
\item
The ring $K_{\fn\cap K}$ is Gorenstein for every $\fn\in\Max S$.
\item
The homomorphism $\sigma$ has finite flat dimension.\qed
 \end{enumerate}
\end{ssubcorollary}

\begin{ssubnotes}
 \label{notes:Gdimhom}
A notion of finite G-dimension that applies to arbitrary
local homomorphisms is defined in \cite{AF:qG}.  Proposition
\ref{prop:relative-semidualizing} and \cite[4.3, 4.5]{AF:qG}
show that the definitions agree when both apply; thus, Corollary
\ref{cor:relative-semidualizing} recovers \cite[4.4.1, 4.4.2]{AF:qG}.
 \end{ssubnotes}

\subsection{Relative rigidity}
\label{Relative rigidity}
Proposition \ref{prop:relative-semidualizing} and Theorem \ref{thm:rigidity} yield:

\begin{ssubtheorem}
\label{thm:relativerigidity}
Assume that $\sigma$ has finite G-dimension.

A complex $M$ in $\dcatb S$ is $D^\sigma$-rigid if and only if it is  
isomorphic to $D^{\sigma}_a$ for some idempotent $a\in S$; such an  
idempotent is uniquely defined.
   \qed
 \end{ssubtheorem}

This theorem greatly strengthens some results of \cite{YZ2},
where rigidity is defined using a derived version of
Hochschild cohomology, due to Quillen: There is a functor
 \[
\Rhom{S\dtensor{K}S}S{-\dtensor K-}\col\dcat[S]\times\dcat[S]\to\dcat[S]\,,
 \]
see \cite[\S3]{AILN} for details of the construction, which has the following 
properties:

 \begin{ssubchunk}
   \label{quillenHH}
Quillen's \emph{derived Hochschild cohomology modules}, see \cite[\S3]{Qu}, are given by
  \[
\Ext n{S\dtensor{K}S}S{M\dtensor KN}=
\HH{-n}{\Rhom{S\dtensor{K}S}S{M\dtensor KN}}\,.
  \]
 \end{ssubchunk}

 \begin{ssubchunk}
   \label{flat-derived}
When $S$ is $K$-flat one can replace $S\dtensor{K}S$ with $S\otimes_KS$;
see \cite[Remark 3.4]{AILN}.
 \end{ssubchunk}

 \begin{ssubchunk}
   \label{reduction}
When $\fd\sigma$ is finite, for every complex $M$ in $\dcatb S$ with 
$\fd_KM<\infty$ and for every complex $N$ in $\dcat[S]$, by 
\cite[Theorem 4.1]{AILN} there exists an isomorphism
 \[
\Rhom{S\dtensor{K}S}S{M\dtensor KN}
\simeq\Rhom S{\Rhom SM{D^{\sigma}}}N
\quad\text{in}\quad \dcat\,.
 \]
   \end{ssubchunk}

Yekutieli and Zhang \cite[4.1]{YZ1} define $M$ in $\dcat$ to be \emph{rigid
relative to  $K$} if $M$ is in $\dcatb S$, satisfies $\fd_KM<\infty$, and 
admits a \emph{rigidifying isomorphism}
  \[
\mu\col M\xra{\,\simeq\,}\Rhom{S\dtensor{K}S}S{M\dtensor KM}
\quad\text{in}\quad \dcat\,.
  \]
By \ref{flat-derived}, when $K$ is a field, this coincides with the notion 
introduced by Van den Bergh \cite[8.1]{VdB}.  On the other hand,
\eqref{eq:rigid} and \ref{reduction}, applied with $N=M$, give:

   \begin{ssubchunk}
   \label{ch:YZrigidity}
When $\fd\sigma$ is finite, $M$ in $\dcatb S$ is rigid relative to $K$ if 
and only if $\fd_KM$ is finite and $M$ is $D^\sigma$-rigid.
  \end{ssubchunk}

{From} Theorems \ref{thm:relativerigidity} and \ref{thm:gorring} we now 
obtain:
 
\begin{ssubtheorem}
\label{thm:relativerigidityFD}
Assume that $K$ is Gorenstein and $\fd\sigma$ is finite.

The complex $D^\sigma$ then is pointwise dualizing for $S$ and is rigid relative to $K$.

A complex $M$ in $\dcatb S$ is rigid relative to $K$ if and only if $D^{\sigma}_a\cong M$ holds for some idempotent $a$ in $S$. More precisely, when $\delta$ and $\mu$ are rigidifying isomorphisms  for $D^\sigma$ and $M$, respectively, there exists a commutative diagram
 \[
\xymatrixrowsep{2.5pc}
\xymatrixcolsep{4.5pc}
\xymatrix{ 
D^\sigma_a
  \ar@{->}[r]^-{\delta_a}_-{\simeq}
  \ar@{->}[d]^-{\simeq}_{\alpha}
  &\, \Rhom{S\dtensor{K}S}S{D^\sigma_a\dtensor KD^\sigma_a}
\ar@{->}[d]_-{\simeq}^{\Rhom{S\dtensor KS}S{\alpha\dtensor K\alpha}}
  \\
M\ar@{->}[r]_-{\mu}^-{\simeq}
&\,\Rhom{S\dtensor{K}S}S{M\dtensor KM}
}
 \]
where both the idempotent $a$ and the isomorphism $\alpha$ are 
uniquely defined.
 \qed
  \end{ssubtheorem}

In \cite{YZ2} the ring $K$ is assumed regular of finite Krull dimension. This implies $\fd_KM<\infty$ for all $M\in\dcatb S$, so $\fd\sigma<\infty$ holds, and also that $S$ is of finite Krull dimension, since it is essentially of finite type over $K$. Therefore \cite[1.1(a), alias 3.6(a)]{YZ2} and \cite[1.2, alias 3.10]{YZ2} are special cases of Theorem \ref{thm:relativerigidityFD}.  

There also is a converse, stemming from \ref{dualizing_classical} and Theorem \ref{thm:gorring}.

Finally, we address a series of comments made at the end of
\cite[\S3]{YZ2}; they are given in quotation marks, but notation and 
references are changed to match ours.

\begin{ssubnotes}
     \label{rem:YZsurprise}
The paragraph preceding \cite[3.10]{YZ2} reads: ``Next comes a 
surprising result that basically says `all rigid complexes are dualizing'.  
The significance of this result is yet unknown.''  It states:  If $K$ and 
$S$ are regular, $\dim S$ is finite, and $S$ has no idempotents other
that $0$ and $1$, then a rigid complex is either zero or dualizing. 

Theorem \ref{thm:rigidity} provides an explanation of this phenomenon:
 Under these conditions $S$ has finite global dimension, hence every 
semidualizing complex is dualizing.
 \end{ssubnotes}

\begin{ssubnotes}
     \label{rem:YZquotes}
Concerning \cite[3.14]{YZ2}:
``The standing assumptions that the base ring $K$ has finite global 
dimension seems superfluous.''  See Theorem 
\ref{thm:relativerigidityFD}.  

``However, it seems necessary for $K$ to be Gorenstein---see 
\cite[Example 3.16]{YZ2}.''  Compare Theorems \ref{thm:relativerigidity}
and \ref{thm:relativerigidityFD}.

``A similar reservation applies to the assumption that $S$ is regular in 
Theorem 3.10 (Note the mistake in \cite[Theorem 0.6]{YZ1}: there too 
$S$ has to be regular).''  Theorem \ref{thm:relativerigidityFD} shows 
that the regularity hypothesis can be weakened significantly.
 \end{ssubnotes}

\subsection{Quasi-Gorenstein homomorphisms}
\label{Quasi-Gorenstein homomorphisms}

The map $\sigma$ is said to be \emph{quasi-Goren\-stein} if in \ref{eq:reduce1} 
for each $\fn\in\Max S$ the $Q_{\fn\cap Q}$-module $S_\fn$ has finite 
G-dimension and satisfies $\Rhom{Q_{\fn\cap Q}}{S_\fn}{Q_{\fn\cap S}}
\simeq\Shift^{r(\fn)}S_\fn$ for some $r(\fn)\in\BZ$; see \cite[5.4, 6.7, 7.8, 8.4]{AF:qG};
when this holds  $\sigma$ has finite $G$-dimension by Corollary~\ref{cor:localGdim}. 

By part (i) of the next theorem, quasi-Gorensteinness is a property of $\sigma$,
not of the factorization.  The equivalence of (ii) and (iii) also follows from
\cite[2.2]{AI:mm}.

\begin{ssubtheorem}
\label{thm:qgorenstein}
The following conditions are equivalent:
\begin{enumerate}[\quad\rm(i)]
  \item
$D^{\sigma}$ is invertible in $\dcat[S]$.
  \item[\rm(i$'$)]
$D^{\sigma}$ is derived $S$-reflexive in $\dcat[S]$
and $\gdim\sigma<\infty$.
  \item
$\sigma$ is quasi-Gorenstein.
  \item
$\Ext{}PSP$ is an invertible graded $S$-module.
   \end{enumerate}
\end{ssubtheorem}

\begin{proof}
(i)$\iff$(i$'$). This results from Proposition~\ref{prop:relative-semidualizing} and Corollary~\ref{cor:reflsemid}.

(i)$\iff$(iii).
By \ref{eq:basechange}, one has $D^{\sigma}\simeq\Shift^d\Rhom PSP\dtensor P{\Omega^d_{P|K}}$ in $\dcat$.  It implies that $D^\sigma$ is invertible in $\dcat$ 
if and only if $\Rhom PSP$ is.  By Proposition~\ref{prop:semid}, the latter condition 
holds if and only if $\Ext{}PSP$ is invertible.

(i$'$) \& (iii)$\implies$(ii).  Indeed, for every $\fn\in\Spec S$ the 
finiteness of $\gdim\sigma$ implies that of $\gdim_{P_{\fn\cap P}}{S_\fn}$,
and the invertibility of $\Ext{}PSP$ implies an isomorphism
$\Rhom{P_{\fn\cap P}\!}{S_\fn}{P_{\fn\cap P}}\simeq\Shift^{r(\fn)}S_\fn$
for some $r(\fn)\in\BZ$, see Proposition~\ref{prop:semid}.

(ii)$\implies$(iii).  This follows from Proposition~\ref{prop:semid}.
 \end{proof}

A quasi-Gorenstein homomorphism $\sigma$ with $\fd_KS<\infty$ is said to be \emph{Gorenstein}, see \cite[8.1]{AF:qG}.  When $\sigma$ is flat, it is Gorenstein if and only if for every $\fq\in\Spec S$ and $\fp=\fq\cap K$
the ring $(K_{\fp}/{\fp}K_{\fp})\otimes_K S$ is Gorenstein; see~\cite[8.3]{AF:qG}.
The next result uses derived Hochschild cohomology; see \ref{quillenHH}.
For flat $\sigma$ it is proved in \cite[2.4]{AI:mm}.

 \begin{ssubtheorem}
   \label{thm:gorenstein}
The map $\sigma$ is Gorenstein if and only if $\fd\sigma$ is finite and the graded $S$-module $\Ext{}{S\dtensor KS}S{S\dtensor KS}$ is invertible.  When $\sigma$ is Gorenstein one has
 \[
D^{\sigma}\simeq\Ext{}{S\dtensor KS}S{S\dtensor KS}^{-1}\quad\text{in}\quad \dcat[S]\,,
 \]
and one can replace $S\dtensor KS$ with $S\otimes_KS$ in case $\sigma$ is flat.
 \end{ssubtheorem}

 \begin{proof}
We may assume that $\fd\sigma$ is finite.  One then gets an isomorphism
 \begin{equation}
   \label{eq:specialization}
\Rhom S{D^{\sigma}}S\simeq
\Rhom{S\dtensor{K}S}S{S\dtensor KS}
\quad\text{in}\quad \dcat[S]\tag{\ref{thm:gorenstein}.1}
\end{equation}
from \ref{reduction} with $M=S=N$.  The following equivalences 
then hold:
 \begin{xxalignat}{2}
\sigma \text{ is Gorenstein} 
&\iff
D^{\sigma}\text{ is invertible}
&&[\text{by Theorem \ref{thm:qgorenstein}}]
  \\
&\iff
\Rhom S{D^{\sigma}}S\text{ is invertible}
&&[\text{by Proposition \ref{prop:semid}}]
  \\
&\iff
\Rhom{S\dtensor{K}S}S{S\dtensor KS}\text{ is invertible}
&&[\text{by \eqref{eq:specialization}}]
  \\
&\iff
\Ext{}{S\dtensor KS}S{S\dtensor KS} \text{ is invertible}
&&[\text{by Proposition \ref{prop:semid}}]
 \end{xxalignat}
When $D^\sigma$ is invertible, \eqref{eq:specialization} and
\ref{ch:projhom} yield isomorphisms
 \[
(D^{\sigma})^{-1}\simeq
\Rhom{S\dtensor{K}S}S{S\dtensor KS}
\simeq
\Ext{}{S\dtensor KS}S{S\dtensor KS}
\quad\text{in}\quad \dcat[S]\,,
\]
whence the desired expression for $D^\sigma$. The last assertion comes from \ref{flat-derived}.
 \end{proof}

Combining Theorem \ref{thm:gorenstein}, Proposition~\ref{prop:naturality2}, 
and the isomorphism in \ref{eq:basechange}, we see that $D^\sigma$ can be 
computed from factorizations through arbitrary Gorenstein 
homomorphisms---not just through essentially smooth ones, as provided 
by \ref{ch:independence}.  

\begin{ssubcorollary}\pushQED{\qed}
 \label{cor:gormaps}
If $K\xra{\varkappa}Q\xra{\varkappa'}S$ is a factorization of $\sigma$ with $\varkappa$ Gorenstein and $\varkappa\>'$ finite, then there is an isomorphism
 \[
D^\sigma\simeq\Rhom QSQ\otimes_Q\Ext{}{Q\dtensor KQ}Q{Q\dtensor KQ}^{-1}
\quad\text{in}\quad \dcat[S]\,.
   \qedhere
 \]
 \end{ssubcorollary}

\appendix

\section{Homological invariants}
\label{sec:pbseries}

Let $R$ be a commutative noetherian ring. 

Complexes of $R$-modules have differentials of degree $-1$.  Modules are identified with complexes 
concentrated in degree zero. For every graded $R$-module $H$ we set
  \[
\inf H=\inf\{n\in\BZ\mid H_n\ne0\}
  \quad\text{and}\quad
\sup H=\sup\{n\in\BZ\mid H_n\ne0\}\,.
  \]
The \emph{amplitude} of $H$ is the number $\amp H= \sup H-\inf H$. Thus $H=0$ is equivalent to $\inf H=\infty$; to $ \sup H=-\infty$; to $\amp H=-\infty$, and also to $\amp H<0$.

We write $\dcat[R]$ for the derived category of $R$-modules, and $\Shift$ for its translation functor.  Various full subcategories of $\dcat[R]$ are used in this text. Our notation for them is mostly standard:  the objects of $\dcatP R$ are the complexes $M$ with $\inf\hh M>-\infty$, those of $\dcatN R$ are the complexes 
$M$ with $\sup\hh M<\infty$, and $\dcatB R=\dcatP R\cap\dcatN R$.  Also,  $\dcatfg R$ is the category of complexes $M$ with $\HH nM$ finite for each $n\in\BZ$, and we set $\dcatp R=\dcatfg R\cap\dcatP R$, etc.  

For complexes $M$ and $N$ in $\dcat[R]$ we write $M\dtensor RN$ for the derived tensor product, $\Rhom RMN$ for the derived complex of homomorphisms, and set
  \[
\Tor nRMN=\HH{n}{M\dtensor RN}
  \quad\text{and}\quad
\Ext nRMN=\HH{-n}{\Rhom RMN}\,.
  \]

Standard spectral sequence arguments give the following well known
assertions:

\begin{chunk}
\label{derived-functors}
For all complexes $M$ and $N$ in $\dcat[R]$ there are inequalities
  \begin{align*}
\sup{\hh{\Rhom RMN}}&\le\sup{\hh N}-\inf\hh{M}\,.
  \\
\inf{\hh{M\dtensor RN}}&\ge\inf{\hh M}+\inf\hh{N}\,.
  \end{align*}

If $M$ is in $\dcatp R$ and $N$ is in $\dcatn R$, then $\Rhom RMN$ 
is in $\dcatn R$.

If $M$ and $N$ are in $\dcatp R$, then so is $M\dtensor RN$.
  \end{chunk}

For  ease of reference, we list some canonical isomorphisms:

  \begin{chunk}
 \label{localization}
Let $\fm$ be a maximal ideal of $R$ and set $k=R/\fm$.   For all complexes $M$ in $\dcat[R]$ and $N$ in $\dcatn R$ there are isomorphisms of graded $k$-vector spaces
  \begin{gather*}
k\dtensor{R}{M}\cong (k\dtensor{R}{M})_\fm\cong k\dtensor{R}{M_\fm}
\cong k\dtensor{R_\fm}{M_\fm}\,;
  \\
\Rhom{R}kN\cong\Rhom{R}kN_\fm\cong\Rhom{R}k{N_\fm}
\cong\Rhom{R_\fm}k{N_\fm}\,.
  \end{gather*}
  \end{chunk}

We write $(R,\fm,k)$ \emph{is a local ring} to indicate that $R$ is a  
commutative noetherian ring with unique maximal ideal $\fm$ and 
with residue field $k=R/\fm$.

The statements below may be viewed as partial converses to those in 
\ref{derived-functors}.

  \begin{chunk}
 \label{finite:test}
Let $(R,\fm,k)$ be a local ring and $M$ a complex in $\dcatfg R$.

If $\Rhom RkM$ is in $\dcatN R$, then $M$ is in $\dcatN R$.

If $k\dtensor RM$ is in $\dcatP R$, then $M$ is in $\dcatP R$.
  
See \cite[2.5, 4.5]{FI} for the original proofs.  The proof of \cite[1.5]{AI:mm}
gives a shorter, simpler, argument for the second assertion; it can be 
adapted to cover the first one.
   \end{chunk}

Many arguments in the paper utilize invariants of local rings with values
in the ring $\BZ[\![t]\!][t^{-1}]$ of formal Laurent  series in $t$ with
integer coefficients.  The \emph{order} of such a series
$F(t)=\sum_{n\in\BZ}a_nt^n$ is the number
  \begin{align*}
\ord(F(t))&=\inf\{n\in\BZ\mid a_n\ne0\}\,.
  \end{align*}

To obtain the expressions for Poincar\'e series and Bass series in Lemmas
\ref{lem:poincare} and \ref{lem:bass} below, we combine ideas from Foxby's
proofs of \cite[4.1, 4.2]{Fx} with the results in \ref{finite:test};
this allows us to relax some boundedness conditions in \cite{Fx}.

  \begin{chunk}\textit{Poincar\'e series.}
    \label{Poincare series}
For a local ring $(R,\fm,k)$ and for $M$ in $\dcatp R$, in view of
\ref{derived-functors} the formula below defines a formal Laurent series,
called the \emph{Poincar\'e series} of $M$:
  \[
P^R_M(t)=\sum_{n\in\BZ}\rank _k\Tor nRkM\,t^n\,.
  \]

\begin{subchunk}
 \label{minimal}
When $(R,\fm,k)$ is a local ring, each complex $M\in\dcatp R$ admits a resolution $F\tiso M$ with $F\in\dcatp R$, such that  $\dd(F) \subseteq \fm F$ holds and each $F_n$ is free of finite rank; this forces $\inf F=\inf{\hh M}$.  Since $k\otimes_{R}F$ is a complex of $k$-vector spaces with zero differential, there are  
isomorphisms
  \[
k\dtensor RM\simeq k\otimes_{R}F\simeq\hh{k\otimes_{R}F}
  \quad\text{in}\quad \dcat[R]\,,
 \]
which imply equalities $\rank_k\Tor nRkM=\rank_RF_n$ for all $n\in\BZ$.
 \end{subchunk}

In \ref{supp:poincare} and Lemma~\ref{lem:poincare} below the ring $R$ is not assumed local.

\begin{subchunk}
\label{supp:poincare}
For $M$ in $\dcatp R$ and $\fp$ in $\Spec R$ the conditions $\fp\in 
\Supp M$
and $P^{R_\fp}_{M_\fp}(t)\ne0$ are equivalent; when they hold one has
$\ord(P^{R_\fp}_{M_\fp}(t))=\inf\hh{M_\fp}$.

Indeed, both assertions are immediate consequences of \ref{minimal}.
  \end{subchunk}

\begin{sublemma}
\label{lem:poincare}
Let $M$ and $N$ be complexes in $\dcatfg R$ and $\fp$ be a prime
ideal of $R$.

If $(M\dtensor RN)_\fp$ is in $\dcatP{R_\fp}$, then so are $M_{\fp}$
and $N_{\fp}$, and there are equalities
  \begin{align*}
P^{R_{\fp}}_{(M\dtensor RN)_\fp}(t)
&=P^{R_{\fp}}_{M_{\fp}}(t)\cdot P^{R_{\fp}}_{N_{\fp}}(t)\,,
\\
\inf \hh{(M\dtensor RN)_\fp}
&=\inf \hh{M_\fp}+\inf \hh{N_\fp}\,.
  \end{align*}
\end{sublemma}

\begin{proof}
In $\dcat[R_{\fp}]$ one has $(M\dtensor RN)_{\fp}\simeq
M_{\fp}\dtensor{R_{\fp}}N_{\fp}$, so it suffices to treat the case
when $(R,\fp,k)$ is local.  Note the following isomorphisms of graded
vector spaces:
  \begin{align*}
\hh{k\dtensor R(M\dtensor RN)}
         &\cong\hh{(k\dtensor RM)\dtensor k (k\dtensor RN)}\\
          &\cong\hh{k\dtensor RM}\otimes_{k} \hh{k\dtensor RN}
  \end{align*}
The hypotheses and \ref{derived-functors} yield $\HH n{k\dtensor
R(M\dtensor RN)}=0$ for $n\ll0$, so the isomorphism implies that
$k\dtensor RM$ and $k\dtensor RN$ are in $\dcatP R$, and thus $M$ and $N$
are in $\dcatp R$ by \ref{finite:test}.  When they are, for each $n\in\BZ$
one has an isomorphism of $k$-vector spaces
  \begin{align*}
(\hh{k\dtensor RM}\otimes_k\hh{k\dtensor RN})_n
 &\cong\bigoplus_{i+j=n}\HH i{k\dtensor RM}\otimes_{k}\HH j{k\dtensor
 RN}\\
  &\cong\bigoplus_{i+j=n}\Tor iRkM\otimes_{k}\Tor jRkN\,.  \end{align*}
By equating the generating series for the ranks over $k$, we get the
desired equality of Poincar\'e series; comparing orders and using
\ref{supp:poincare} gives the second equality.
 \end{proof}
   \end{chunk}

  \begin{chunk}\textit{Bass series.}
For a local ring $(R,\fm,k)$ and for $N$ in $\dcatn R$, in view of
\ref{derived-functors} the following formula defines a formal Laurent
series, called the \emph{Bass series} of $N$:
  \[
I_R^N(t)=\sum_{n\in\BZ}\rank _k\Ext nRkN\,t^n\,.
  \]

\begin{subchunk}
  \label{Bass series}
For a local ring $R$ and $N$ in $\dcatn R$ one has
$\ord(I_R^N(t))=\depth_RN$; this follows from the definition of depth,
see Section~\ref{Depth}.
Furthermore, $I_{R}^{N}(t)$ is a Laurent polynomial if and only if $N$
has finite injective dimension; see, for example, \cite[5.5]{AF:hd}.
\end{subchunk}

In the remaining statements the ring $R$ is not necessarily local.

\begin{subchunk}
\label{supp:bass}
For $N$ in $\dcatn R$ and $\fp$ in $\Spec R$ the conditions 
$\fp\in \Supp N$ and $I_{R_\fp}^{N_\fp}(t)\ne0$ are equivalent; 
when they hold one has $\ord(I_{R_\fp}^{N_\fp}(t))=\depth_{R_\fp}\!N_\fp$.

Indeed, in view of \ref{Bass series} the assertions follow from the fact 
that $\depth_{R_\fp} N_\fp< \infty$ is equivalent to $\hh{N_\fp}\ne0$; 
see, for instance, \cite[2.5]{FI}.
  \end{subchunk}

\begin{sublemma}
\label{lem:bass}
Let $M$ and $N$ be complexes in $\dcatfg R$ and $\fp$ a prime ideal  
of $R$.

If $\Rhom RMN$ is in $\dcatN R$ then $M_{\fp}$ is in $\dcatp{R_{\fp}}$.

If, in addition, $\fp$ is the unique maximal ideal of $R$, or $\fp$  
is maximal and
$N$ is in $\dcatn R$, or $M$ is in $\dcatp R$ and $N$ is in $\dcatn R 
$, then
there are equalities
  \begin{align*}
I_{R_{\fp}}^{{\Rhom RMN}_{\fp}}(t)
&=P^{R_{\fp}}_{M_{\fp}}(t)\cdot I_{R_{\fp}}^{N_{\fp}}(t)\,,
\\
\depth_{R_\fp}\!({\Rhom RMN}_{\fp})
&=\inf(\hh{M_\fp})+\depth_{R_\fp}\!(N_\fp)\,.
  \end{align*}
   \end{sublemma}

\begin{proof}
Assume first that $\fp$ is  maximal and set $k=R/\fp$.  One gets
isomorphisms
\begin{align*}
\hh{\Rhom Rk{\Rhom RMN}}
        & \cong \hh{\Rhom R{k\dtensor RM}N} \\
         &\cong \hh{\Rhom k{k\dtensor RM}{\Rhom RkN}} \\
         &\cong \Hom k{\hh{k\dtensor RM}}{\hh{\Rhom RkN}}
  \end{align*}
of graded $k$-vector spaces by using standard maps.  In view of \ref{derived-functors}, for $n\gg0$ one has  $\HH n{\Rhom Rk{\Rhom RMN}}=0$, so the isomorphisms yield $k\dtensor RM\in\dcatP R$ and $\Rhom RkN\in\dcatN R$.  When $R$ is local, one gets $M\in\dcatp R$ and $N\in\dcatn R$ from \ref{finite:test}.  For general $R$,
this implies $M_{\fp}\in\dcatp{R_{\fp}}$ in view of the isomorphism $k\dtensor RM \simeq k\dtensor{R_\fp}{M_\fp}$ from \ref{localization}. If $N$ is in $\dcatn R$, then by referring once more to \emph{loc.~cit.}
we can rewrite the isomorphisms above in each degree $n$ in the form
  \begin{align*}
\Ext{n}{R_\fp}k{\Rhom RMN_\fp}
&\cong\Hom k{\Tor{}{R_\fp}k{M_\fp}}{\Ext{}{R_\fp}k{N_\fp}}_{-n}\\
&\cong\bigoplus_{i-j=n}\Hom k{\Tor i{R_\fp}k{M_\fp}}{\Ext{-j}{R_\fp}k 
{N_\fp}}\,.
  \end{align*}
For the generating series for the ranks over $k$ these isomorphisms give
   \begin{align*}
I_{R_{\fp}}^{{\Rhom RMN}_{\fp}}(t)
        &=\bigg(\sum_{i\in\BZ}\rank_k \Tor i{R_\fp}k{M_\fp} t^i\bigg)
             \bigg(\sum_{j\in\BZ}\rank_k\Ext{j}{R_\fp}k{N_\fp}t^j 
\bigg)\\
      &=P^{R_{\fp}}_{M_{\fp}}(t)\cdot I_{R_{\fp}}^{N_{\fp}}(t)\,.
       \end{align*}
Equating orders of formal Laurent and using \ref{supp:bass} one gets the second equality.

Let now $\fp$ be an arbitrary prime ideal and $\fm$ a maximal
ideal containing $\fp$.  The preceding discussion shows that
$M_{\fm}$ is in $\dcatp{R_{\fm}}$, hence $M_{\fp}$ is in 
$\dcatp{R_{\fp}}$.  When $M$ is in $\dcatp R$ and $N$ is in $\dcatn R$ 
one has ${\Rhom RMN}_{\fp}\cong\Rhom{R_\fp}{M_\fp}{N_\fp}$, so the
desired equalities follow from those that have already been established.
 \end{proof}
    \end{chunk}

\begin{chunk}
\label{support}
The \emph{support} of a complex $M$ in $\dcatb R$ is the set
 \[
\Supp_RM=\{\fp\in\Spec R\mid \hh M_\fp\ne0 \}\,.
 \]
One has $\Supp_{R}M=\varnothing$ if and only if $\hh M=0$, if and only
if $M\simeq0$ in $\dcat[R]$.

For all complexes $M,N$ in $\dcatb R$ there are equalities
\[
\Supp_{R}(M\dtensor RN) = \Supp_{R}M\cap \Supp_{R}N =
\Supp_{R}\Rhom RMN\,.
\]
This follows directly from \ref{supp:poincare}, \ref{supp:bass},  and Lemmas~\ref{lem:poincare} and \ref{lem:bass}.
  \end{chunk}

\end{document}